\documentclass[12pt]{amsart}
\pdfoutput=1
\usepackage{amsmath,amsthm,amssymb}
\usepackage{tensor}
\usepackage{graphicx}
\usepackage{tikz-cd}
\usepackage{longtable}

\newcommand{\slt}[1]{\text{SL}_2(#1)}
\newcommand{\Zi}{\mathbb{Z}[i]}

\newcommand{\mfq}{\mathfrak{q}}
\newcommand{\mfp}{\mathfrak{p}}
\newcommand{\mtr}[1]{\text{tr}(#1)}
\newcommand{\mtrs}[1]{\text{tr}^2(#1)}

\newcommand{\GN}[1]{\text{N}(#1)}
\newcommand{\Det}[1]{\text{det}(#1)}
\newcommand{\ind}[1]{\mathbf{1}_{\{ #1\}}}

\newcommand{\adm}[3]{\tensor*[_#1]{\Gamma}{_#2^#3}}
\newcommand{\ad}[2]{{\Gamma}{_#1^#2}}
\newcommand{\dm}[2]{{}_#1{\Gamma}{^{#2}}}
\newcommand{\round}[1]{\lfloor#1 \rceil}

\newtheorem{proposition}{Proposition}
\newtheorem{theorem}{Theorem}
\newtheorem{lemma}{Lemma}

\newtheorem{corollary}{Corollary}

\usepackage[margin=1in]{geometry}

\begin{document} 
\title[Low-lying Geodesics]{Low-lying Geodesics in an Arithmetic Hyperbolic Three-Manifold} 
\author{Katie McKeon} 
\date{7/1/2019}
\begin{abstract}
We examine closed geodesics in the quotient of hyperbolic three space by the discrete group of isometries SL(2,Z[i]).  There is a correspondence between closed geodesics in the manifold, the complex continued fractions originally studied by Hurwitz, and binary quadratic forms over the Gaussian integers.  According to this correspondence, a geodesic is called fundamental if the associated binary quadratic form is.  Using techniques from sieve theory, symbolic dynamics, and the theory of expander graphs, we show the existence of a compact set in the manifold containing infinitely many fundamental geodesics.
\end{abstract}

\maketitle

\tableofcontents

\section{Introduction}

\subsection{Closed Geodesics on the Modular Surface}\

In \cite{duke}, Duke showed that closed geodesics on the modular surface equidistribute when grouped by discriminant, see also \cite{clozelullmo}.  In \cite{ELMV} Einsiedler, Lindenstrauss, Michel, and Venkatesh gave a modern treatment of Linnik's approach to this problem using the ergodic method.   A key step is ruling out other potential weak-$*$ limits of closed geodesics.  This raises the basic question on what other weak-$*$ limits could arise.  Because the geodesic flow is a shift map (see Chapter~\ref{pollicott}) this question is trivial without more restrictions.  They asked whether there is an infinite collection of closed geodesics having fundamental discriminant and being trapped in a compact subset of the modular surface.  That is they don't visit the cusp, or are ``low-lying.''  Bourgain and Kontorovich \cite{bk2017} showed an abundance of fundamental low-lying geodesics on the modular surface, answering the question above in a quantitative sense. 

We will attack the corresponding problem in the Picard $3$-manifold $\slt{\Zi} \backslash \mathbb{H}^3$.  However, the solution is not as simple as applying the machinery from \cite{bk2017} to a `thin semi-group' with well-established growth properties.  For one, continued fractions in the complex plane are much more complicated to work with versus simple continued fractions on the real line.  In particular, our symbolic encoding of closed geodesics does not display a semigroup structure because the shift map is restricted.  We also have to develop in this setting much of the machinery (Chapters~\ref{pollicott}-\ref{delta}) which was already available to \cite{bk2017} for the modular surface.

\subsection{The Main Theorem}\

We must establish some terminology before stating the main result.  Consider the upper half space model of hyperbolic three-space: 
$$\mathbb{H}^3 = \{z + tj : z \in \mathbb{C}, t \in \mathbb{R}_+ \}$$
equipped with the line element 
$$ds^2 = (dz_1^2+d z_2^2 + dt^2)/t^2$$
where an element $\begin{pmatrix} a & b \\ c & d \end{pmatrix}$ in the group of isometries $G=PSL_2{\mathbb{C}}$ acts by
$$
z+tj \mapsto  \frac{(az+b) \overline{(cz+d)}}{|cz+d|^2 + |c|^2 t^2} +  \left(\frac{t}{|cz+d|^2+|c|^2t^2}\right) j .
$$
When $\mathbb{H}^3$ is embedded as a subset of Hamilton's quaternions, this expression simplifies as $z+tj \mapsto (a(z+tj)+b)(c(z+tj)+d)^{-1}$.  Since $G$ extends to a simply transitive action on the frame bundle $F\mathbb{H}^3$, we can identify an element in $G$ with where it moves some representative reference frame.  We can also identify $\mathbb{H}^3 \leftrightarrow G / SU(2)$ and
$$
T^1 \mathbb{H}^3 \leftrightarrow G  / SO(2).
$$
Geodesic flow on $T^1 \mathbb{H}^3$ under this identification is represented by right-multiplication by the one parameter group generated by $a_t = \begin{pmatrix} e^t & 0 \\ 0 & e^{-t} \end{pmatrix}$.  There is a correspondence between conjugacy classes of primitive hyperbolic matrices in $\Gamma = \slt{\Zi}$ and closed geodesics.

To restrict to ``low-lying'' geodesics,we only consider those in the standard fundamental domain lying in a certain region $\{z + tj : t<R\}$.  See for example, the figure below which depicts the standard fundamental domain for $\Gamma \backslash \mathbb{H}^3$ and the cutoff $t<2.5$.
\begin{figure}[h!]
\includegraphics[scale=.5]{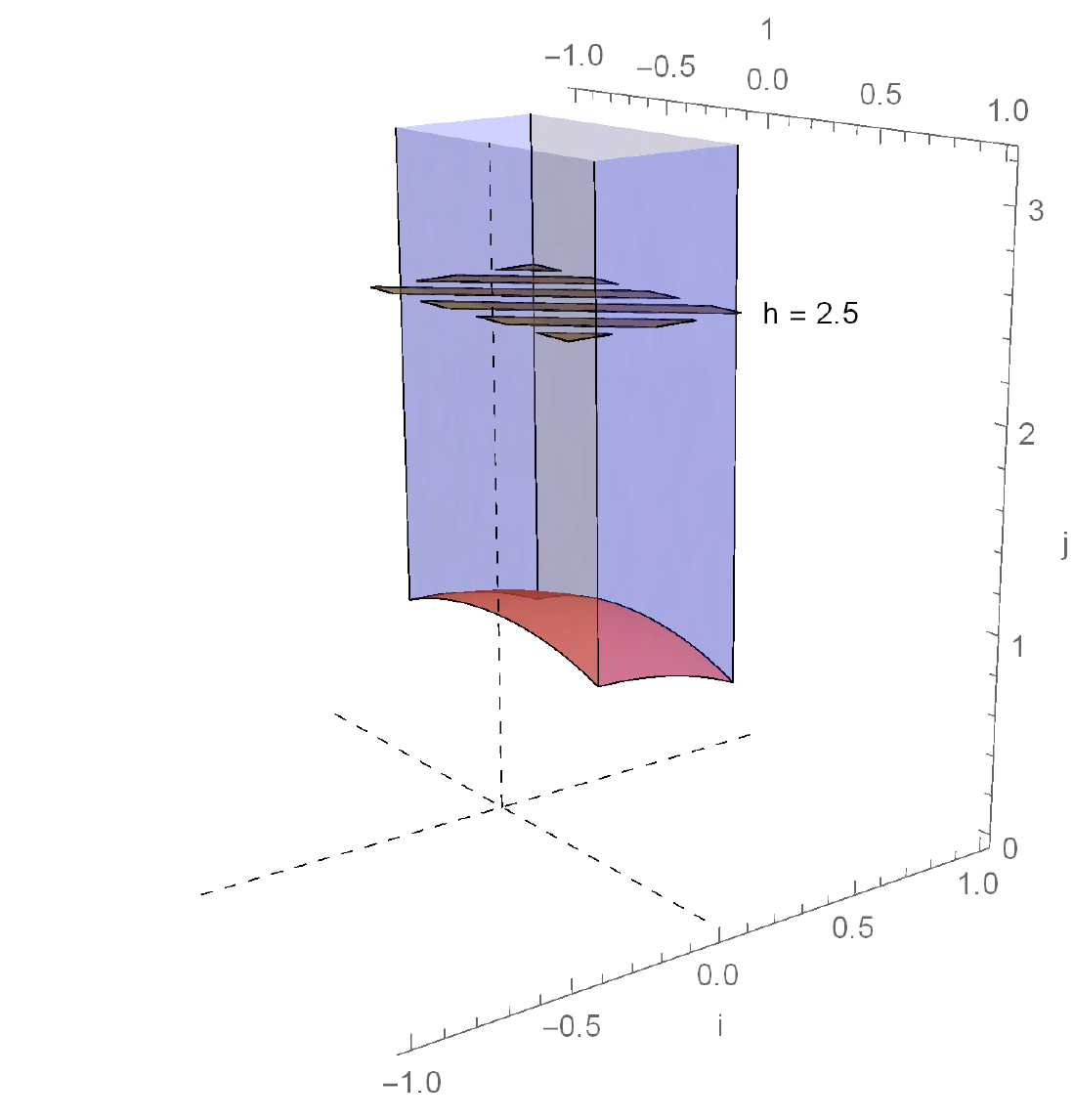}
\centering
\caption{Height Cutoff $t<2.5$}
\end{figure}

Once we have developed a symbolic encoding of closed geodesics, it is trivial to manufacture infinitely many geodesics which are low-lying by enumerating periodic points (with restricted orbits) of a map analogous to the Gauss map for continued fractions.  Adding the condition that geodesics be fundamental places the requirement that the geodesic $g$, as a loxodromic element of $\slt{\Zi}$, has $\mtrs{g}-4$ which satisfies certain conditions (see Section~\ref{fd}), a sufficient one being that it is square-free.  If $C_D$ is the set of distinct geodesics with discriminant $D$ and $N(\cdot)$ denotes the norm of a Gaussian integer, then we will show the following quantitative result:
\begin{theorem}
\label{main}
For any $\epsilon>0$, there is a compact region $Y(\epsilon) \subset \Gamma \backslash \mathbb{H}^3$ and a set $D(\epsilon)$ of fundamental discriminants such that
\begin{align*}
\#\{ D \in \mathcal{D}(\epsilon): \GN{D}<X \} &\gg_\epsilon X^{1-\epsilon}, \hspace{.3in} X \to \infty
\end{align*}
and for all $D \in \mathcal{D}(\epsilon)$, 
\begin{align*}
\#\{ \gamma \in C_D : \gamma \subset Y(\epsilon)\} &> |C_D|^{1-\epsilon}.
\end{align*}
\end{theorem} 
Of course, qualitatively, this solves the problem of producing infinitely many fundamental low-lying geodesics in the Picard $3$-manifold.

\subsection{Strategy of the Proof}\

Our three main tools come from symbolic dynamics, expander graphs, and sieve theory. 

\subsubsection{Thermodynamic Formalism and Renewal Theorems}\

After converting the study of closed geodesics to periodic points of an analogue of the `Gauss map', we are led to study the dynamics of a subshift of finite type.  In particular, define
$$
\Sigma_R := \{ (x_1, x_2, \ldots ) : x_i \in P_R, A_{x_i,A_{i+1}}=1 \text{ for all } i\}
$$
where $P_R$ is some finite alphabet (see Chapter~\ref{pollicott}), and $A$ is a $P_R \times P_R$ binary matrix conveying transition rules (i.e. $A_{x,y}=1$ if $x$ can be followed by $y$ and $A_{x,y}=0$ otherwise). The shift map $\sigma$ on $\Sigma_R$ is defined as $\sigma( (x_1, x_2, \ldots)) = (x_2, x_3 , \ldots)$.  Choosing the appropriate $P_R$ and $A$ following the work of \cite{PolGauss}, gives an essentially one-to-one correspondence between closed geodesics lying in a compact set (corresponding to the choice of $R$) and periodic points of $\Sigma_R$ under $\sigma$.    Denote the set of closed geodesics associated to $(\Sigma_R, \sigma)$ via this correspondence as $\Gamma_R \subset \slt{\Zi}$ and
$$
B_X := \left\{ \begin{pmatrix} a & b \\ c & d \end{pmatrix} \in \slt{\Zi} : ||a||^2+||b||^2+||c||^2 + ||d||^2 < X^2 \right\}.
$$
The methods of Lalley in \cite{Lalley} are straightforward to apply to our situation and lead to the following:
\begin{theorem}
\label{normball}
For fixed $R \geq 3$, there is a $\delta_R \in (0,2)$ so that
$$\# ( \Gamma_R \cap B_X ) \asymp X^{2\delta_R}$$
as $X \to \infty$.
\end{theorem}
It will be crucial in a later argument that the set above is large.  
\begin{theorem}
\label{deltabig}
The growth parameter given by Theorem~\ref{normball} satisfies
$$
\lim_{R \to \infty} \delta_R \to 2.
$$
\end{theorem}

Finally, we will need local information about closed geodesics.  This is where expander graphs are used crucially.  We follow the work of \cite{BGS} together with \cite{BKM}.  For $q \in \Zi$, set $\slt{q} = \slt{\Zi}/ (q)$.
\begin{theorem}
\label{affinesieve}
For each $R>8$, there is some absolute spectral gap $\Theta_R>0$ and absolute constants $c_R, C_R > 0$ such that for all square-free $q \in \Zi$ and $\omega \in \slt{q}$ we have the estimate
$$
\left| \#\{g \in \Gamma_R \cap B_X : g \equiv \omega \bmod q\} - \frac{\#(\Gamma_R \cap B_X)}{|\slt{q}|} \right| \ll_R \#(\Gamma_R \cap B_X) E(q,X)
$$
as $X \to \infty$, where 
$$
E(q,X) = \begin{cases} e^{-c_R \sqrt{\log X}} & \GN{q} < C_R \log X \\ \GN{q}^{C_R} X^{-\Theta_R} & \GN{q}> C_R\log X \end{cases}.
$$
\end{theorem}

\subsubsection{A First Attempt via the Affine Sieve}\

Theorem~\ref{affinesieve} allows us to follow the affine sieve procedure (see \cite{bgs10}, \cite{sgs}, or \cite{k14}) up to a level of distribution $X^\alpha$ where the exponent $\alpha_R = \Theta_R/C - o(1)$ and show for sufficiently large $\delta_R$ that
\begin{align*}
\#\{ g \in \Gamma_R \cap B_X : p \text{ a Gaussian prime}, p| (\mtrs{g}-4) \implies \GN{p}>X^{\alpha_R}\} \\
\gg X^{2{\delta_R}-o(1)}.
\end{align*}
Since $\mtrs{g}-4$ factors as $(\mtr{g}+2)(\mtr{g}-2)$, any Gaussian prime dividing $\mtrs{g}-4$ must divide one of its linear factors.  This fact allows us to conclude from the ``almost prime'' estimate above that 
$$
\#\{g \in \Gamma_R \cap B_X : (\mtrs{g}-4) \text{ not square-free}\} \ll X^{4-\alpha_R}.
$$
See the proof of Theorem~\ref{sqfr} on page~\pageref{sqfr} for details.  

Comparing the two estimates, we would have our main result if only we could show that 
$$2\delta_R > 4 - \Theta_R/C.$$
It seems that making $\delta_R$ near 2 (as in Theorem~\ref{deltabig}) suffices.  However, the shared dependence of $\delta_R$ and $\Theta_R$ on $R$ cannot at this time be decoupled, so this attack fails, and we must use something more than expansion.

\subsubsection{Beyond Expansion}\

To have stronger control on the exponent of distribution, we create bilinear (in fact, multilinear) forms, replacing $\Gamma_R \cap B_X$ by a specially constructed subset $\Pi$, see Chapter~\ref{siftingset}.  We analyze the set
$$
\#\{ \omega \in \Pi : \mtrs{\omega}-4 \equiv 0 \bmod q\}
$$
via abelian harmonic analysis (on $\Zi / (q)$). The characters of small order up to some intermediate level $Q_0$ can be handled by expansion.  The characters of larger order are now dealt with by appealing to bilinear forms techniques, namely Cauchy-Schwarz and estimating exponential sums.  Our methods allow us to sieve up to the absolute level of distribution $\alpha=1/16 - \epsilon$ from which Theorem~\ref{main} follows.

\section{Closed Geodesics and Dirichlet Forms}

We express the upper half space model of hyperbolic $3$-space as a subset of Hamilton's quaternions, i.e.
$$\mathbb{H}^3 = \{z + tj: z \in \mathbb{C}, t \in \mathbb{R}_+\}$$
where $i^2=j^2=-1$ and $ij=-ji$.  The group $PSL_2(\mathbb{C})$ acts on $\mathbb{H}^3$ by
$$
\begin{pmatrix} a & b \\ c & d \end{pmatrix} \cdot (z+tj) := (a(z+tj)+b)(c(z+tj)+d)^{-1}.
$$
The inverse above should be interpreted as the Hamiltonian inverse.  Considering the action of $PSL_2(\mathbb{C})$ on the boundary $\hat{\mathbb{C}}$ gives a correspondence with M\"{o}bius transformations.

\subsubsection{The Picard Group $PSL_2(\Zi)$}\

Inside of $G=PSL_2(\mathbb{C})$, we have the discrete subgroup $\Gamma = PSL_2(\Zi)$, sometimes referred to as the Picard group.  We can express
\begin{align*}
\Gamma = \left\langle  
  \begin{pmatrix} 1 & 1 \\ 0 & 1 \end{pmatrix}, \begin{pmatrix} 1 & i \\ 0 & 1 \end{pmatrix},
    \begin{pmatrix} 0 & -1 \\ 1 & 0 \end{pmatrix},
    \begin{pmatrix} -i & 0 \\ 0 & i \end{pmatrix} \right\rangle
\end{align*}
and this allows us to write the fundamental Dirichlet domain for the quotient $\Gamma \backslash \mathbb{H}^3$ as 
$$
\mathcal{F} := \{ (z,t) | Re(z) \in [-1/2, 1/2], Im(z) \in [0,1/2], |z|^2+t^2>1\}.
$$
\begin{figure}[h!]
\includegraphics[scale=.4]{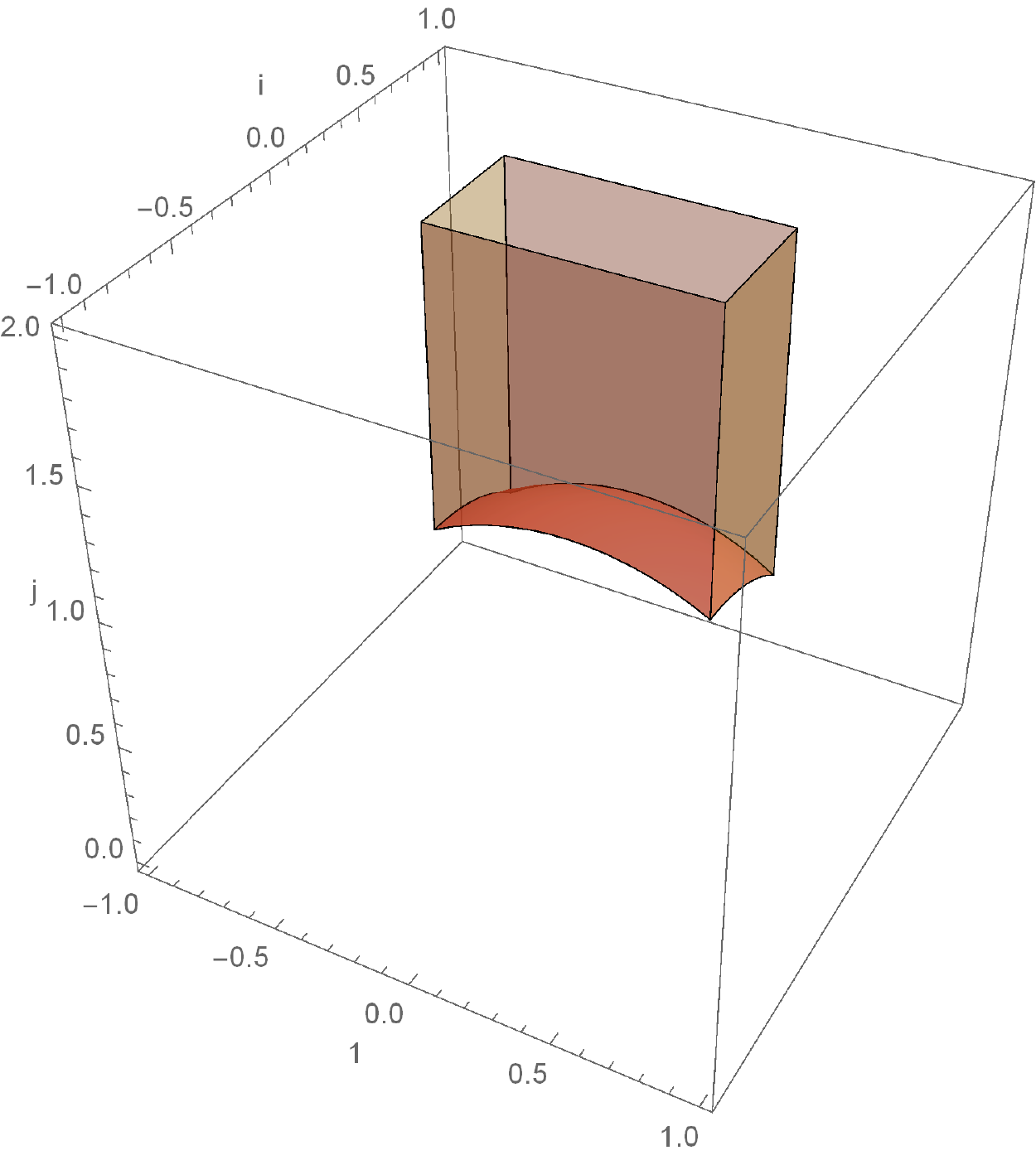}
\centering
\caption{Fundamental Domain for $\slt{\Zi} \backslash \mathbb{H}^3$}
\end{figure}

\subsubsection{Closed Geodesics}\
Write 
$$M = \left\{ \begin{pmatrix} e^{i\theta} & 0 \\ 0 & e^{-i\theta} \end{pmatrix}: \theta \in [0,2\pi)\right\}.$$  
A geodesic which is closed in $\Gamma \backslash SL_2(\mathbb{C})/M$ is identified by $(\Gamma g M) a_l = (\Gamma g M)$.  Since $a_l$ commutes with $M$, this translates to
$$
gMa_l = AgM
$$
for some $A \in \Gamma$, $g \in G$.  More explicitly, we may write
\begin{align*}
A &= g a_l m g^{-1} \\
&= g \begin{pmatrix} e^{l/2} & 0 \\ 0 & e^{-l/2} \end{pmatrix} \begin{pmatrix} e^{i\theta/2} & 0 \\0 & e^{-i\theta/2} \end{pmatrix} g^{-1}
\end{align*}
for some $m \in M$.

\subsubsection{Properties of Closed Geodesics}\
\begin{itemize}
\item Trace and Eigenvalues:
    
    By the equation above, we infer that $A$ is diagonalized by $g$ with eigenvalues $\exp(\pm (l/2+i\theta/2))$.  Following Beardon, we call $A$ hyperbolic if $\mtrs{A} \in [4,\infty)$ and strictly loxodromic if $\mtrs{A} \not\in[0,\infty)$.  The term loxodromic (referring to transformations with $2$ fixed points in $\hat{\mathbb{C}}$) encompasses both.  The trace of $A$ is 
    $$
    \mtr{A} = 2 \cosh(l/2+i\theta/2).
    $$
    On the other hand, we can express the eigenvalue as
    $$
    e^{l/2+i\theta/2} = \frac{\mtr{A} + \sqrt{\mtrs{A}-4}}{2},
    $$
    and this gives us a way to determine the closed geodesic associated to an arbitrary loxodromic transformation in $PSL_2(\Zi)$.
    
\item Length and Primitivity:
    
    The length of the geodesic is given by $l$.  Note that any powers of $A$ satisfy $A^k = g (a_lm)^k g^{-1}$, which would suggest that their length is $kl$.  However, $A^k = (ga_lmg^{-1}) \cdots (ga_lmg^{-1})$ indicates that we are traversing the same closed geodesic $k$ times.  So, we call $A$ primitive when it is not a power of another element in the Picard group.
    
  \item Equivalent geodesics:
    
    Technically, a closed geodesic is an element of the quotient $SL_2(\Zi) \backslash SL_2(\mathbb{C})/M$ satisfying $(\Gamma g M) a_l = (\Gamma g M)$.  We chose a particular representative $A$, but any conjugation $BAB^{-1}$ where $B \in \Gamma$ would give the same geodesic.  Hence, geodesics are equivalent if they are in the same $\Gamma$ conjugacy class.
    
    \item Fixed points, visual points:
    
    The loxodromic transformation $A$ has two fixed points in $\mathbb{C}$ which can be found by solving 
    $$
    z = \frac{az+b}{cz+d}.
    $$
    In other words, the fixed points are roots of the (homogenized) binary quadratic form $Q_A(1,z) = cz^2+(d-a)z-b$ with coefficients in $\Zi$.  Solving this, we get
    $$
    \alpha = \frac{(a-d) + \sqrt{\mtrs{A}-4}}{2c}, \hspace{.3in} \overline{\alpha} = \frac{(a-d) - \sqrt{\mtrs{A}-4}}{2c}.
    $$
    On the other hand, if we have $g$ we can calculate
    $$
    \lim_{l \to \infty} ga_l.j = \alpha.
    $$
    A similar statement (as $l \to -\infty$) gives the reverse direction of the geodesic.  These points $\alpha,\overline{\alpha}$ are referred to as the visual points of the geodesic.

\end{itemize}

\subsubsection{Dirichlet Forms}\

Many of the observations above suggest a correspondence between closed geodesics and binary quadratic forms with coefficients in $\Zi$, also known as Dirichlet forms.  We have
\begin{align*}
    A= \begin{pmatrix} a & b \\ c & d \end{pmatrix} \mapsto Q_A(x,y) = cx^2+(d-a)xy-by^2 
\end{align*}
modulo the greatest common divisor of $c$, $d-a$, and $b$ and up to choice of unit.  This correspondence is explained further in \cite{sarnak}.  We associate the discriminant of the Dirichlet form $D_A = \mtrs{A}-4$ with the closed geodesic corresponding to $A$.

\subsubsection{Fundamental Discriminants for Dirichlet Forms}\
\label{fd}

Note that a discriminant $D$ of a Dirichlet form must be a square mod 4 and hence $D \bmod 4 \in \{0,1,-1,2i\}$.   Moreover, each $D \in \mathbb{Z}[i]$ with a square residue mod $4$ is a discriminant of some form. We call a discriminant $D$ fundamental if it cannot be expressed as $D = q^2 D_0$ where $q$ is a non-unit and $D_0$ is also a discriminant.  This is equivalent to another other common definition which states that $D$ is fundamental if any form $Q(x,y) = ax^2+bxy+cy^2$ with discriminant $D$ must be primitive (i.e. $(a,b,c)=1$.)    Note that this also agrees with the work of Hilbert, i.e. that $D$ is fundamental if and only if $|D|$ is the relative discriminant of the extension of $\mathbb{Q}(i,\alpha)$ over $\mathbb{Q}(i)$.

A geodesic is fundamental if its associated discriminant is fundamental.  We will sieve down to geodesics with square-free discriminant, only catching the $D = \pm 1 \bmod 4$ case.

\section{A Symbolic Encoding of Geodesics}
\label{pollicott}

We now describe the results of Pollicott in \cite{PolGauss}.  Denote a circle of radius $1$ about a center $z$ as $C(z)$.  Consider the region $\mathcal{X}$ in $\mathbb{H}^2$ exterior to the three circles $C(i), C(1), $ and $C(-1)$ where we have the removed vertical lines $\{z: \Im(z) = k/2: k \in \mathbb{Z}\}$ and horizontal lines $\{z : \Re{z} = k/2 : k \in \mathbb{N}\}$. The region $\mathcal{X}$ is shown in blue in the figure below. 

\begin{figure}[h]
\includegraphics[scale=.8]{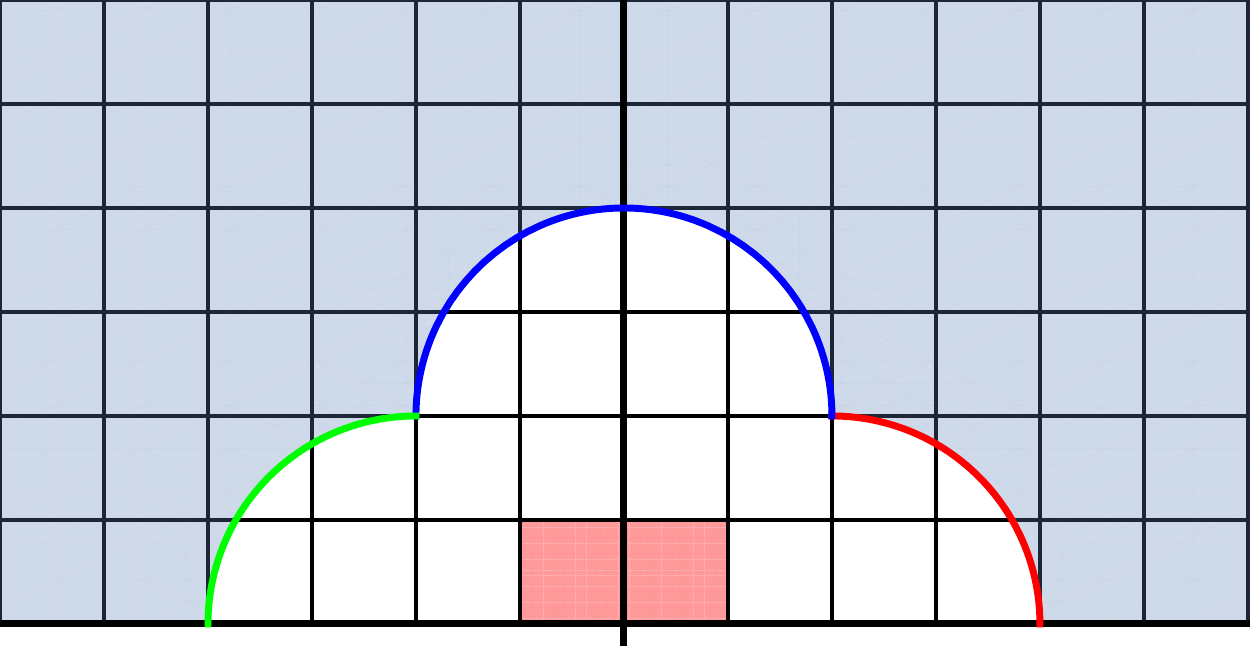}
\centering
\caption{The Fundamental Region $\mathcal{X}$}
\end{figure}

For each $z \in \mathcal{X}$ there is a unique closest (in Euclidean distance) Gaussian integer, which we denote $\round{z}$. The map $\hat{f}$ will move $z$ to the unit square centered about the origin by first subtracting $\round{z}$.  If $\Im(z-\round{z})<0$, then $f$ rotates about the origin by $\pi$.  Finally, the involution $\omega \mapsto -1/\omega$ is performed to return to $\mathcal{X}$.  Formally, we define $\hat{f}$ as
\begin{align*}
\hat{f}(z) = \begin{cases} \frac{(-1)}{z - \round{z}} & \Im(z)>\Im(\round{z}) \\
                     \frac{1}{z - \round{z}}    & \Im(z)<\Im(\round{z})\end{cases}.
\end{align*}

Define $S:z \mapsto -1/z$ and $f:= S \circ \hat{f} \circ S$ on $S(\mathcal{X})$.  Note that $S(\mathcal{X})$, the image of $\mathcal{X}$ under $S$, is contained within $\{x+iy : |x|<1/2, 0<y<1/2\}$.  Pollicott proves the following:
\begin{theorem}
There is a bijection between closed geodesics and the periodic points of $f$.
\end{theorem}

Furthermore, \cite{PolGauss} shows that $(\mathcal{X},f)$ admits a Markov partition.  This will allow us to study the system via a simplified encoding.  The naive choice of partition of $\mathcal{X}$ into connected regions between the grid lines is almost correct.  We must introduce two more circles, $C(1+i)$ and $C(-1+i)$ and separate each region intersecting the boundary of either of the circles.  The figure below illustrates the partition which we denote $\mathcal{P}$.  Each connected region after removing $C(1+i)$, $C(-1+i)$ and the $1/2$-spaced grid is a part in the partition.
\begin{figure}[h]
\includegraphics[scale=.8]{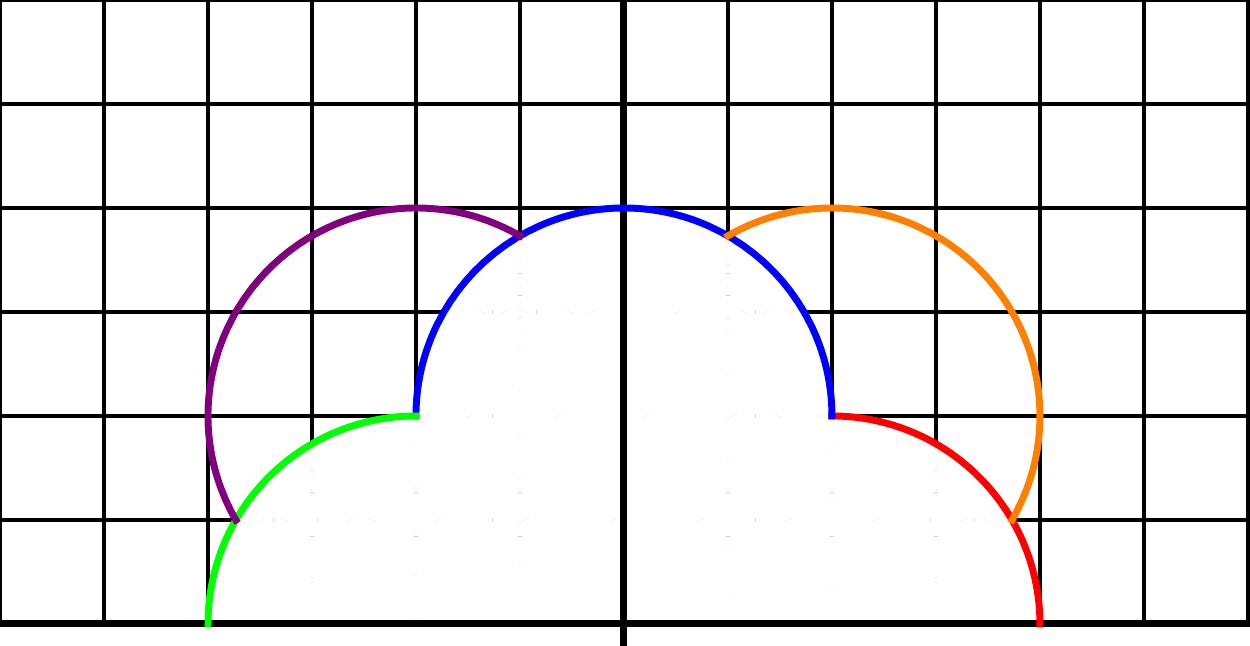} 
\centering
\caption{The Parts in $\mathcal{P}$}
\end{figure}

For each part $\rho \in S(\mathcal{P})$, we associate a distinct label $p$.  The set of all labels is the alphabet for our shift map.  Since $S(\mathcal{P})$ is a Markov partition, $(\Sigma, \sigma)$ gives a symbolic representation of $(S(\mathcal{X}),f)$ where
$$
\Sigma:= \{ (p_1, p_2, \ldots ) : p_i \leftrightarrow \rho_i \in S(\mathcal{P}), f(\rho_i) \supset \rho_{i+1} \}
$$
and $\sigma$ is the shift operator, i.e. $\sigma(p_1, p_2,p_3, \ldots ) = (p_2, p_3, \ldots)$.  More precisely, for any finite admissible word $a = (p_1 \ldots p_n)$ (meaning $a$ occurs as a subword of some $\alpha \in \Sigma$) we define the cylinder set $C_a$ as
$$
C_a = \bigcap_{k=1}^n f^{-k}(\rho_k).
$$
Since $|f'(z)| \leq \frac{1}{2}$, we have that the diameter of $C_\alpha$ is at most $\frac{1}{2^n}$.  We also have $C_{(p_1, \ldots, p_n)} \subset C_{(p_1, \ldots, p_{n-1})}$.  Therefore the map  
\begin{align*}
\pi : \Sigma &\to S(\mathcal{X}) \\
    (p_1, p_2, \ldots) &\mapsto \bigcap_{k=1}^\infty C_{(p_1, \ldots, p_k)}
\end{align*}
is well-defined.  Since the interiors of distinct cylinders of length $n$ (i.e. the cylinder defined on a word of length $n$) are disjoint, $\pi$ is one-to-one.  The image of $\pi$ is all of $S(\mathcal{X})$ up to a set of Lebesgue measure $0$ (the orbit of the grid under $f$ must be removed) and the following diagram commutes:
\[ \begin{tikzcd}
\Sigma \arrow{r}{\sigma} \arrow[swap]{d}{\pi} &\Sigma \arrow{d}{\pi} \\%
S(\mathcal{X}) \arrow{r}{f}& S(\mathcal{X})
\end{tikzcd}
\]
In particular, we have a bijection between the periodic points of $\sigma$ and the periodic points of $f$.  This allows us to study closed geodesics by analyzing the system $(\Sigma, \sigma)$.

From the correspondence between periodic orbits in $(S(\mathcal{X}), f)$ and closed geodesics, we can measure how close a geodesic is to the cusp in $\slt{\Zi} \backslash \slt{\mathbb{C}}$ by determining the minimum distance from origin to a point in the orbit of the associated periodic point.  So define $\mathcal{P}_R \subset \mathcal{P}$ by only allowing parts from $\mathcal{P}$ if they lie completely inside the ball of radius $R$ centered at the origin.  Recall that the parts $\mathcal{P}$ were defined for the conjugate system $(\mathcal{X}, \hat{f})$ and so the images of parts in $\mathcal{P}_R$ under $S$ fall outside a small ball centered at the origin.  The figure below has the parts included in $\mathcal{P}_4$ shaded.
\begin{figure}[h]
\includegraphics[scale=.5]{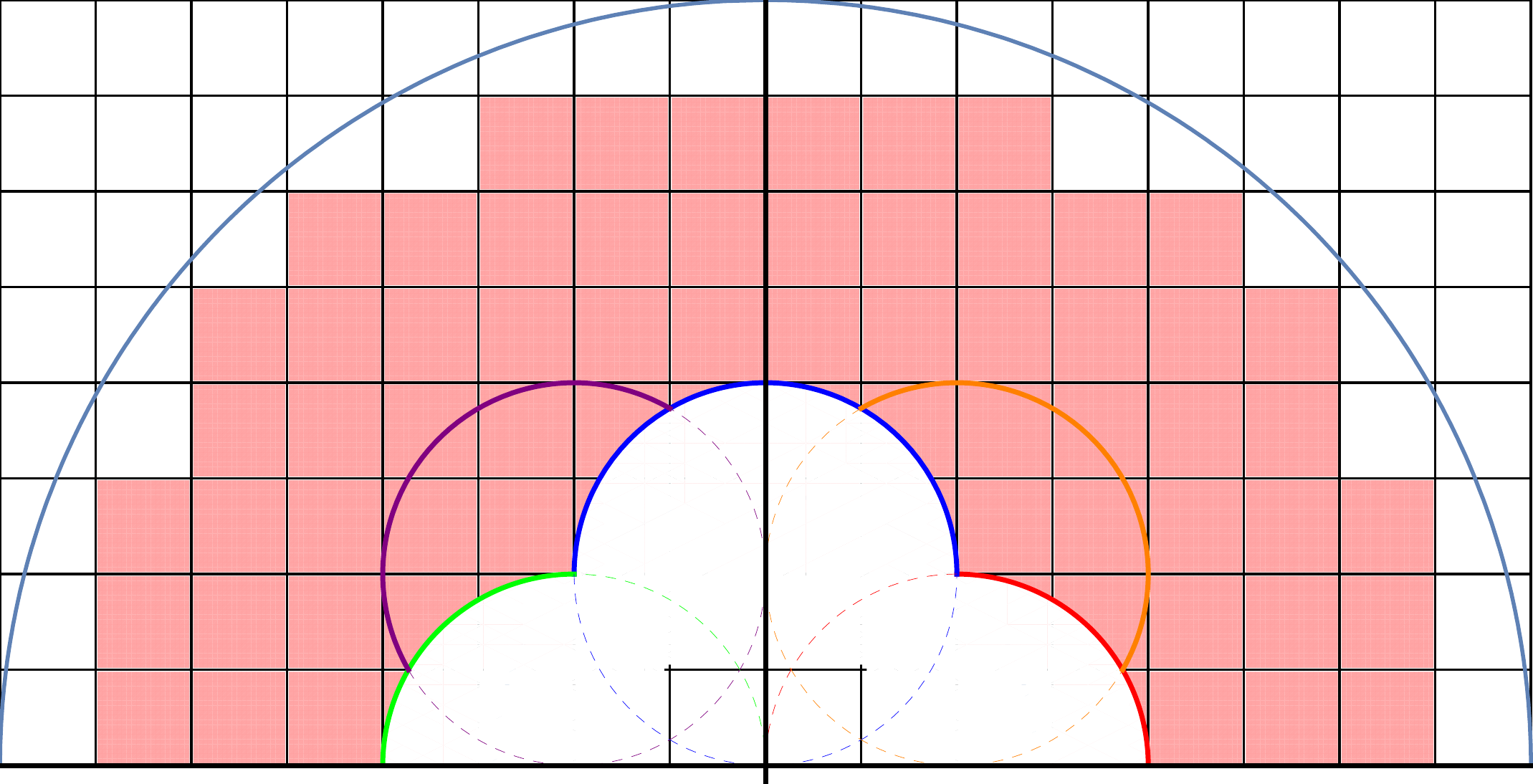}
\centering
\caption{Parts in $P_4$}
\end{figure}

 Our corresponding symbolic encoding is 
$$
\Sigma_R:= \{ (p_1, p_2, \ldots ) : p_i \leftrightarrow \rho_i \in \mathcal{P}_R, f(\rho_i) \supset \rho_{i+1} \}.
$$
The new system $(\Sigma_R, \sigma)$ is now a subshift of finite type as its alphabet is finite.  The problem of counting low-lying geodesics translates into a problem of counting periodic points of $\sigma$ in $\Sigma_R$.

\section{Counting Geodesics via Thermodynamic Formalism}
\label{lall}

The asymptotic for geodesics derived in this chapter is a straightforward application of the work of Lalley in \cite{Lalley}.  We summarize the ideas which give rise to the method in the next section towards an asymptotic for geodesics counted by congruence classes.

For $a \in \Sigma$, we define the distortion function $\tau(a) := \log|f'(\pi(a))|$.  The $N$-th Birkhoff sum for the distortion function is then $S_N \tau(a) := \sum_{k=0}^{N-1} \tau(\sigma^k a)$ where $S_0(a) = 0$.  If $\sigma^N(a) = a$ and in particular $a$ is a visual point corresponding to a closed geodesic, then $S_N \tau(a) = -l$ where $l$ is the length of the geodesic.  This gives some  indication that the following function will be useful in counting geodesics. For $T>0$ and $a \in \Sigma_R$ Lalley's renewal function is defined as 
$$
N(T,a) := \sum_{k=0}^\infty \sum_{b : \sigma^k(b)=a} g(y) \ind{S_k\tau(b) \leq T}.
$$
Partitioning the sum by the preimage $\sigma^{-1}(a)$, we arrive at the following recursive relation, known as the renewal equation
$$
N(T,a) = g(a) \ind{a \leq T} + \sum_{b: \sigma(b)=a} N(T-\tau(b), b).
$$
In order to analyze the renewal function, we are led to study its Laplace transform and a certain transfer operator.

For a continuous function $f$ defined on $\Sigma_R$ and $0<\rho<1$, we define
\begin{align*}
\text{var}_n (f) &:= \sup\{ |f(a)-f(b)| : a_i = b_i \text{ for all } 0 \leq i \leq n \}, \\
|f|_\rho &:= \sup_{n \geq 0} \frac{\text{var}_n(f)}{\rho^n}.
\end{align*}
Then $\mathcal{F}_\rho:=\{f : |f|_\rho<\infty\}$ is the space of H\"{o}lder continuous functions which is a Banach space with norm $||\cdot ||_\rho = |\cdot|_\rho + || \cdot||_\infty$.  The transfer operator, depending on $s \in \mathbb{C}$, acts on $\mathcal{F}_\rho$ as follows:
$$
\mathcal{L}_s f(a) := \sum_{\substack{b \\ \sigma(b)=a}} e^{s \tau(b)} f(b).
$$

First note that $\mathcal{L}_s f(a)$ is a bounded linear operator.  Additionally, when $s \in \mathbb{R}$ the coefficients in the sum are all positive.  For real $s$, would like to compare the spectrum of $\mathcal{L}_s$ to that of an positive matrix.  In particular, we would like to apply an analogue of the Perron-Frobenius theorem.  In order to do so, we must establish a few more properties of the system $(\Sigma_R, \sigma)$.

Recall that a matrix $A$ is irreducible if for each position $(i, j)$ there is some power of $A$ such that the $(i,j)$-th entry is positive.  In analogy, we say that $(\Sigma_R, \sigma)$ is irreducible if for each two states $p_1, p_2 \in \mathcal{P}_R$ there is some finite admissible word beginning with $p_1$ and ending in $p_2$.
\begin{lemma}
$(\Sigma_R, \sigma)$ is irreducible as long as $R\geq 4$.
\end{lemma}
\begin{proof}
By our definition of $\Sigma_R$, a finite subword of the form $(p_{i_1}, p_{i_2}, \ldots, p_{i_n})$ must satisfy $\hat{f}(p_{i_k}) \supset p_{i_{k+1}}$.  In other words, we must show that for each part in $p \in \mathcal{P}_R$ that there is some $k$ such that $\mathcal{P}_R \subset \hat{f}^n(p)$.

One may recall from the diagrams in the previous section that for any part $p \in \mathcal{P}_R$, the image under $\hat{f}$ contains at least one of the following regions $(1)-(8):$
\begin{figure}[h!]
\includegraphics[scale=.5]{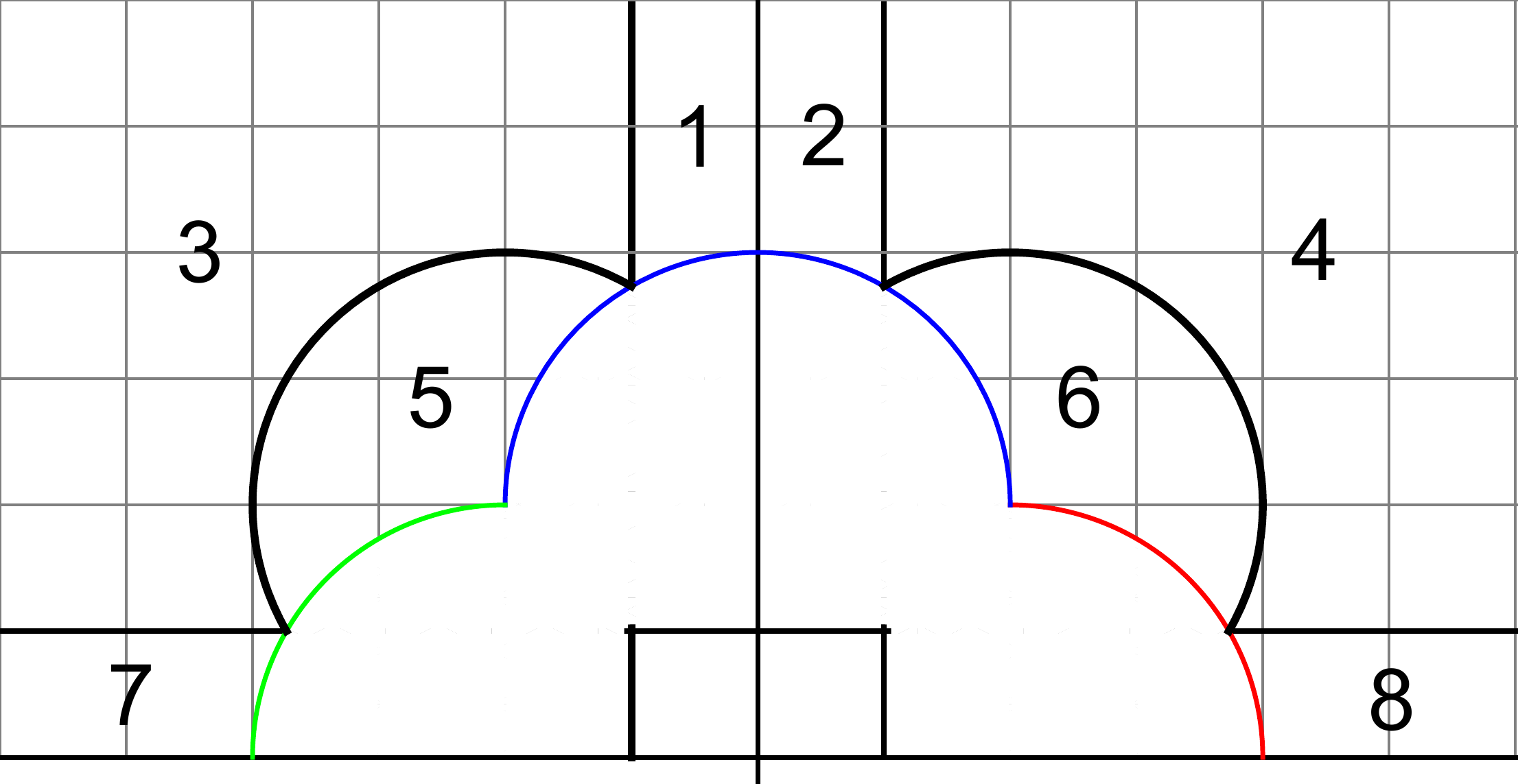}
\centering
\caption{The Main $8$ Regions}
\label{8regions}
\end{figure}
Each region contains at least one square part.  The image of a square part under $\hat{f}$ is either the union of the even regions or the odd regions.  In either case, the next iterate of $\hat{f}$ is the union of all $8$ regions.  Hence $\mathcal{P}_R \subset \hat{f}^3(p)$ for any $p$.  
\end{proof}

A state $p \in \mathcal{P}_R$ is periodic of period $k$ if any finite admissible word beginning and ending in $p$ must have length divisible by $k$.  The period of the system $(\Sigma_R, \sigma)$ is the greatest common factor of the periods of all of its states and a system is said to be aperiodic if this greatest common factor is 1. 
\begin{lemma}
$(\Sigma_R, \sigma)$ is aperiodic as long as $R\geq 4$.
\end{lemma}
\begin{proof}
Since we already established irreducibility, we only need to show that one state is aperiodic.  Take the square part to the left of $3$, i.e. $p = \{x+iy: x \in (2.5,3), y \in (0,.5)\}$.  Its image under $\hat{f}$ is the union of the even regions (refer to the previous figure) and hence $p \subset \hat{f}(p)$.  Therefore, the orbit of a point in $p$ may return to $p$ after any number of iterates of $\hat{f}$.
\end{proof}

\subsection{Properties of the Spectrum of $\mathcal{L}_s$}\

We are now in a position to cite Ruelle's Perron-Frobenius theorem (see \cite{parrypoll} for proof):
\begin{theorem}
\label{RPF}
For $s \in \mathbb{R}$ the spectrum of $\mathcal{L}_s$ has the following properties:
\begin{enumerate}
\item $\mathcal{L}_s$ has a simple maximal positive eigenvalue $\lambda_s$ with corresponding eigenfunction $h_s$ which can be chosen to be positive.
\item The remainder of the spectrum is contained in a disc of radius less than $\lambda_s$.
\item There is a unique probability measure $\mu_s$ on $\Sigma_R$ such that $\mathcal{L}_s^* \mu_s = \lambda_s \mu_s$.
\item $\frac{1}{\lambda_s^n} \mathcal{L}_s^n v \to h_s \int v d\mu_s$ uniformly for  all continuous $v$ if $h_s$ is normalized so that $\int h_s d \mu_s = 1$.
\end{enumerate}
\end{theorem}

Consider $\mathcal{L}_s$ as a family in $s \in \mathbb{R}$ and define the pressure functional as $P(s) = \log \lambda_s$.  The pressure is increasing in $s$ and there is a unique solution $s = \delta_R$ to $P(s) = 0$.  Since $\delta_R$ is featured in our main asymptotic, it will be necessary to determine the dependence of $\delta_R$ on $R$ in a later section.  Our next step is to consider the family $\mathcal{L}_s$ for $s \in \mathbb{C} - \mathbb{R}$.  

\begin{theorem}
\label{nonlattice}
For $s \in \mathbb{C}- \mathbb{R}$, the spectrum of $\mathcal{L}_s$ is  contained in a disc centered at zero of radius $\lambda_{\Re{s}}$ (the maximal eigenvalue of $\mathcal{L}_{\Re{s}}$).
\end{theorem}
\begin{proof}
We summarize the proof provided by Lalley in section 11 of \cite{Lalley}.

There are two cases: either $\tau$ is lattice in which case the spectral radius of $\mathcal{L}_s$ is strictly smaller than the radius on the real axis or it is nonlattice in which case $\mathcal{L}_s$ has spectral radius equal to $\lambda_{\Re{s}}$ at regularly spaced intervals along the vertical line $\Re{z}=\Re{s}$.  See Lalley for the precise definition of a  lattice function.  In order to show that $\tau$ is non-lattice we relate it to a function known to be non-lattice.

 $\tau$ is cohomologous to $g$, a height function on suspension of restricted geodesic flow.

It suffices to show that for periodic $x \in \Sigma$ of period $n$, i.e. $\sigma^n(x)=x$, we have $S_n\tau(x)=S_n g(x)$.  If $k$ is the period of $x$, we will show that both $n$th Birkhoff sums give $n/k$ times the length of the closed geodesic associated with $x$: 
\begin{align*}
t^n(M) &= d_H(j,M.j) - d_H(j, \sigma(M).j) \\
&=d_H(j,M.j)-d_H(j,j) \\
&=d_H(j,ga^lg^{-1}.j) \\
&=d_H(g.j, ga^l.j) = l.
\end{align*}
On the other hand, $M$ also corresponds to a mobius function $m$.  Say $\alpha$ is the fixed point of $m$.
\begin{align*}
m'(\alpha) = \frac{1}{(c\alpha+d)^2} =\frac{1}{(c (e^{l/2} -d)/c +d)^2} = e^{-l}.
\end{align*}
So $\log |m'(\alpha)|=\tau^n(\alpha) = l$.

 If $g$ is lattice then the suspension flow is not mixing for any invariant measure, but \cite{rudolph} proves otherwise.
\end{proof}

\subsection{A Renewal Theorem for the Counting Function $N(a,X)$}\

Perturbation estimates ($|\mathcal{L}_s - \mathcal{L}_z|$ for $|s-z|<\epsilon$) imply that the eigenvalue map $s \mapsto \lambda_s$, the lead eigenfunction map $s \mapsto h_s$ and the invariant measure map $s \mapsto \mu_s$ are all holomorphic functions in a small neighborhood of $\delta_R$.  This leads to the local decomposition
$
(1-\mathcal{L}_s)^{-1} = \frac{h_s}{1-\lambda_s} \mu_s + (1-\mathcal{L}_s')^{-1}
$
where $(1-\mathcal{L}_s')^{-1}$ is a holomorphic family of bounded operators.  The first term in the local decomposition about $\delta_R$ contributes the main term in Lalley's estimate:

\begin{theorem} For any $a \in \Sigma_R$,
$$N(T,a) = h_{\delta_R}(a) e^{T\delta_R} + o(e^{T\delta_R}).$$
\end{theorem}
\begin{proof}
We summarize the proof, ignoring issues of convergence which are addressed in section of $8$ of Lalley.  First, define the following Laplace transform of the renewal function
$$
F(s,a) := \int_{-\infty}^\infty e^{Ts} N(T,a) dT
$$
and input the renewal equation to get
\begin{align*}
F(s,a) &= \int_{0}^\infty g(a)e^{Ts} dT +   \sum_{b: \sigma(b)=b} \int_{-\infty}^\infty e^{Ts} N(T-\tau(b),b) dT \\
&= \frac{g(a)}{1-e^s} +  \sum_{b: \sigma(b)=b} \int_{-\infty}^\infty e^{Ts}e^{-s\tau(b)} N(T,b) dT \\
&= \frac{g(a)}{1-e^s} + \mathcal{L}_s F(s,a).
\end{align*}
So $I - \mathcal{L}_s$ applied to the Laplace transform of the renewal function gives $g(a) (1-e^s)^{-1}$.  Now it is clear how information about the spectrum of $\mathcal{L}_s$ yields information about the renewal function.  In particular, where the resolvent exists we have
$$
(1-e^s)^{-1} (I-\mathcal{L}_s)^{-1} g(a)=  F(s,a).
$$
Theorems \ref{RPF} and \ref{nonlattice} and the decomposition of $(1-\mathcal{L}_s)^{-1}$ in a neighborhood of $\delta_R$ imply
$$
F(s,a) = \frac{C(a)}{z} + G(s,a)
$$
where $G(s,a)$ is holomorphic in $\Re(s) \geq \delta_R$.  Integrating a smoothed version of $F(s,a)$ along a vertical line $\Re{s}>\delta_R$ and pulling the contour to $\delta_R$ we get that the pole in $F(s,a)$ contributes the main term
$$
e^{-T\delta} N(T,a) \to C(a)
$$
as $T \to \infty$.
\end{proof}

\subsection{An Asymptotic for Closed Geodesics} \
\label{finren}

Finally, we relate dynamics on $(\Sigma_R, \sigma)$ to a geodesic count.  Define the set of finite admissible words (plus the empty word $\emptyset$) as
$$
\Sigma^* = \{\emptyset\} \cup \{(p_1, \ldots, p_n) : p_i \leftrightarrow \rho_i \in \mathcal{P}_R, f(\rho_i) \supset \rho_{i+1} \}.
$$
There is a bijection between aperiodic words in $\Sigma^*$ and closed geodesics.  Specifically, to a closed geodesic we associate the element in $\Sigma_R^*$ that corresponds to the period of an orbit under $\sigma$.  A periodic word in $\Sigma^*$ corresponds to the a geodesic traversed multiple times.  We define $\tau^*(a) = d(j, g_a j) - d(j, g_{\sigma(a)} j)$ where $d(P, P)$ denotes hyperbolic distance between any two points in $\mathbb{H}^3$ and $g_a \in \slt{\Zi}$ is the local definition of $f$ restricted to $C_a$.  The identity $2\cosh d(j, g.j) = ||g||^2$ (see \cite{EGM} for a  proof specific to $\mathbb{H}^3$) will eventually lead us to the final counts for matrices in a norm ball.  The shift operator extends in a natural way to act on $\Sigma^*$ however we need to resolve the ambiguity of $\sigma^k$ for words of length less than $k$:
\begin{align*}
\sigma(p_1, p_2, \ldots, p_n) &= (p_2, \ldots p_n) \\
\sigma(p_i) &= \emptyset \\
\sigma(\emptyset) &= \emptyset.
\end{align*}
We are now ready to define the finite version of the renewal function
$$
N^*(T,a) := \sum_{k=0}^\infty \sum_{\substack{b: \sigma^k(b)=a \\ b\neq \emptyset}} g(y) \ind{S_n \tau^*(b) \leq T}.
$$
$N^*(T,a)$ satisfies a renewal equation similar to $N(T,a)$ and so we are tempted to treat the finite renewal function analogously.  However, finding an appropriate Banach space of functions for the transfer operators $\mathcal{L}_s^*$ to act on is elusive.   

In order to model $\Sigma_R^*$ after $\Sigma_R$ we introduce a new state $0$ and for any $(p_1, \ldots, p_k) \in \Sigma_R^*$ append an infinite tail of $0$'s to achieve an infinite word.  The empty word maps to the infinite string of zeros and the action of $\sigma$ is well-defined between the finite model of $\Sigma^*$ and the infinite one.  The space $\mathcal{F}_\rho(\Sigma_R^* \cup \Sigma_R)$ of H\"{o}lder continuous functions satisfies the same properties (with the same norm) as previously.  However, the addition of the new `0'-state means that the system $(\Sigma_R^*\cup \Sigma_R, \sigma)$ is no longer irreducible and hence Ruelle's Perron-Frobenius theorem does not immediately apply to the spectrum of the transfer operators defined on $\mathcal{F}_\rho(\Sigma_R^* \cup \Sigma_R)$.  Since we are only after an asymptotic for $N^*(T,a)$ and long words in $\Sigma_R^*$ may be approximated reasonably well (in the product topology) by words in $\Sigma_R$, we can use what we have already shown about the transfer operators on $\mathcal{F}_\rho(\Sigma_R)$ to prove the following:
\begin{theorem} 
If $h_{s}^*$ is the leading eigenfunction of $\mathcal{L}_s^*$ then
$$N^*(T,a) = h_{\delta_R}^*(a) e^{T\delta_R} + o(e^{T\delta_R}) .$$
\end{theorem}

\begin{proof}
Our first claim is that for $a \in \Sigma_R^*$ and $b \in \Sigma_R$ which are in the same $N$-cylinder in $\Sigma_R^* \cup \Sigma_R$ (i.e. $a_i = b_i$ for $1 \leq i \leq N$) and for $k \ll N$
$$
S_k \tau^*(a) = S_k \tau(b) + O(2^{-N+k})
$$
Let $\overline{a} \in \Sigma_R$ be the periodic word with period $a_1, \ldots a_N$.  We claim that $S_k \tau^*(a) = S_k \tau(\overline{a})$.  Recall from the proof of Theorem~\ref{nonlattice} that for any geodesic with period $c\in \Sigma_R^*$, $S_l \tau^*(c) = S_l \tau(\overline{c}$ where $l$ is the length of $c$.  We can also write
\begin{align*}
S_k \tau^*(a) &= S_N \tau^*(a) - S_{N-k}\tau^*(\sigma^k(a)) \\
&= S_N \tau(\overline{a}) - S_{N-k} \tau(\overline{\sigma^k(a)}) \\
&= \left( S_k \tau(\overline{a}) + S_{N-k} \tau(\overline{\sigma^k(a)}) \right) - S_{N-k} \tau(\overline{\sigma^k(a)})  \\
&= S_k \tau(\overline{a})
\end{align*}
On the other hand, by bounded distortion (see Lemma~\ref{bd} which is more easily understood with the notation in the next section) $S_k \tau(b) = S_k \tau(\overline{a}) + O^{2^{-N+k}}$.  This proves the first claim.

From the previous claim it follows that for $a \in \Sigma_R^*$ and $b \in \Sigma_R$ which are in the same $N$-cylinder and $k \ll N$
$$
N(T-2^{-N+k}, b) \leq N^*(T, a) \leq N(T+2^{-N+k}, b)
$$
for $g \equiv 1$. 

Then iterating the renewal equation $N$ times, one has
\begin{align*}
N^*(T,a) = &\sum_{\substack{b' : \sigma^N(b')=a \\ \sigma^{N-1}(b') \neq \emptyset}} N^*(T-S_N\tau^*(b'), b') \\
&+\sum_{k=1}^{N-1} \sum_{\substack{b : \sigma^k(b)=a \\ \sigma^{N-1}(b) \neq \emptyset}} g(b) \ind{S_N \tau^*(b) \leq T} + g(a)\ind{T \geq 0}.
\end{align*}
As $T \to \infty$ second line does not change.   For each summand in the first line we can find $b' \in \Sigma_R$ to sandwich between the two terms
$$
N(T \pm 2^{-N+k} -S_N \tau^*(b'), b) \asymp h(b) e^{(T-S_N\tau^*(b'))\delta_R} e^{\pm \delta_R 2^{-N+k}}
$$
Send $N \to \infty$ and use the continuity of $h$ to get the statement of the theorem.
\end{proof}

 If $h_s$ is the leading eigenfunction for $\mathcal{L}_s$ the transfer operator in $(\Sigma_R, \sigma)$ then the leading eigenfunction $h^*_s$ of $\mathcal{L}^*_s$ agrees with $h_s$ on $\Sigma_R$.  This is Lemma 6.1 of Lalley.

The theorem follows from a sandwiching argument of the renewal function $N^*(T,x)$ between $N(T,x)$ with appropriate parameters.  A similar argument will appear in the next section.

Combining the asymptotic for $N^*(\sqrt{2 \cosh T},\emptyset)$ with the identity $2\cosh d(j, gj) = ||g||^2$ yields Theorem~\ref{normball} which we restate here:
\begin{theorem}
For fixed $R \geq 3$, there is a $\delta_R \in (0,2)$ so that
$$\# ( \Gamma_R \cap B_X ) \asymp X^{2\delta_R}$$
as $X \to \infty$.
\end{theorem}

\section{Counting Geodesics with Congruence Conditions}
Here we combine the work of Bourgain, Gamburd, Sarnak in \cite{BGS} with the expansion idea of Bourgain, Kontorovich, Magee in \cite{BKM} to estimate
$$
\#\{ g \in \Gamma_R\cap B_X : g \equiv \omega \bmod q \} 
$$
for some $\omega \in \slt{q}$ (recall that $\slt{q}:= \slt{\Zi}/(q)$).

In order to detect congruence classes we introduce a new space of functions $\mathcal{F}_\rho(\Sigma_R \times \slt{q})$.  If $f$, defined on $\Sigma_R \times \slt{q}$, is continuous in each variable, we can define
\begin{align*}
    ||f||_\infty &:= \sup_x \left( \sum_{g \in \slt{q}} |f(x,g)|^2 \right)^{1/2}, \\
    \text{var}_n f &:= \sup \left\{ \left(\sum_{g \in \slt{q}} |f(a,g)-f(b,g)|^2 \right)^{1/2} : a_i = b_i \text{ for all } 0 \leq i \leq n\right\}, \\
    |f|_\rho &:= \sup_n \frac{\text{var}_n f}{\rho^n}.
\end{align*}
Then $\mathcal{F}_\rho(\Sigma_R \times \slt{q}) = \{f \in C(\Sigma_R \times \slt{q}): ||f||\infty<\infty, |f|_\rho<\infty\}$ with norm $||\cdot||_\rho = ||\cdot||_\infty + |\cdot|_\rho$ is a Banach space.

Before defining the action of the transfer operators on this space, we must explain how the action of $(\Sigma_R, \sigma)$ extends to $\slt{q}$.  Recall the definition of $\hat{f}$ in $(\mathcal{X}, \hat{f}):$
\begin{align*}
    \hat{f}(z) = 
    \begin{cases} \frac{-1}{z - \lfloor z \rceil} & \Im(z)>\Im(\lfloor z \rceil) 
    \\ \frac{1}{z - \lfloor z \rceil} & \Im(z)<\Im(\lfloor z \rceil) \end{cases}.
\end{align*}
Locally, i.e. when restricting to the interior of a part $p \in \mathcal{P}$, we may represent the action of $\hat{f}$ as a fractional linear transformation:
$$
\hat{f}|_p = \begin{pmatrix} 0 & -1 \\ 1 & 0 \end{pmatrix} \begin{pmatrix} -i & 0 \\ 0 & i \end{pmatrix}^j \begin{pmatrix} 1 & - \lfloor z \rceil \\ 0 & 1 \end{pmatrix}
$$
where $j \in \{0, 1\}$ reflects whether rotation is necessary for the image to be in the upper half-plane.  Moreover, the inverse of $\hat{f}|_p$ is well-defined and represented by a fractional linear transformation in $\slt{\Zi}$.  

We now introduce some notation and operations on $\Sigma_R$ in order to describe the preimages of an element under $\sigma$.  Let
$$
\Gamma^n := \{ (p_1, \ldots, p_n) : p_i \leftrightarrow \rho_i \in \mathcal{P}_R, f(\rho_i) \supset \rho_{i+1} \}
$$
be the set of admissible words of length $n$. \label{conc} We denote concatenation of two finite words with $\Vert$, i.e. for $a=(a_1, \ldots, a_n) \in \Gamma^n$ and $b=(b_1,\ldots, b_k) \in \Gamma^k$,
$$
a \Vert b = (a_1, \ldots, a_n, b_1, \ldots, b_k)
$$
If $a_n$ and $b_1$ satisfy the subshift rules, then we say the concatenation $a \Vert b \in \Gamma^{n+k}$ is an admissible one.  Concatenations of the form $a \Vert b$ are also well-defined for $b \in \Sigma_R$ as long as $a$ is a finite word.  In order to describe finite words which give admissible concatenations we set
\begin{align*}
\adm{x}{y}{n} := \{a \in \Gamma^n : x \Vert a \Vert y \text{ is admissible}\},\\
\Gamma_y^n:= \{a \in \Gamma^n :  a \Vert y \text{ is admissible}\}.
\end{align*}
While $y$ may be an infinite word, we must have $x$ a finite word for the definition above.  

For $a \in \Sigma_R$, we now write $\sigma^{-1}(a) = \{ b \Vert a : b \in \Gamma^1_a \}$ and for each $b \vert a \in \sigma^{-1}(a)$ denote the inverse branch of $f$ at $b \Vert a$ as $g_b$.  In other words, $g_b \in \slt{\Zi}$ satisfies $g_b(\pi(a)) = \pi(b\Vert a)$ and $f(b\Vert a) = \pi^{-1}\circ g_b^{-1} \circ \pi(b \Vert a) = a.$

We are now ready to describe the congruence transfer operators: \label{to}
$$
M_z f(x,g) = \sum_{a \in \Gamma_{1,x}} e^{z\tau(a \Vert x)} f(a \Vert x; g_a g).
$$

\subsection{Bounding $M_z$ in the Supremum Norm}\

We will exhibit cancellation in the iterates of the transfer operator
$$
M^N_z f(x,g) = \sum_{a \in \Gamma^N_x} e^{zS_N\tau^( a \Vert x)} f(a \Vert x; g_a g)
$$
by treating the prefix and suffix of $a \in \Gamma^N_{x}$ separately.  Let $N = M+R$ and for any $b \in \Gamma^M$ define \label{mub}
$$
\mu_b := \sum_{a \in \adm{b}{x}{R}} e^{z S_N\tau(b\Vert a \Vert x)} \delta_{\pi_q(g_a)} = \sum_{a \in \adm{b}{x}{R}} e^{zS_M\tau(b\Vert a \Vert x)} e^{zS_R\tau(a\Vert x)}\delta_{\pi_q(g_a)}
$$
where $\pi_q: \slt{\Zi} \to \slt{q}$.  In order to decouple $M$ and $R$, we need a lemma which reassures us that the value of $e^{zS_M\tau(b \Vert a \Vert x)}$ does not change much for fixed $b$ and varying $a$.  We first establish a property of the system $(S(\mathcal{X})_R, f)$ called {\it bounded distortion.}

\begin{lemma}
\label{bd}
For fixed $R$ there is some $C=C(R)$ so that for any $p \in \mathcal{P}_R$
$$
\sup_{z \in p} \left|\frac{f''}{f'}(z)\right| < C.
$$
\end{lemma}
\begin{proof}
From the definition
$$
f(z) = (-1)^j \left(\frac{1}{z} - \round{1}{z}\right)
$$
we have that both $|f'(z)|=\frac{1}{|z|^2}$ and $|f''(z)|=\frac{1}{|z|^3}$.  Since we have fixed $R$ all of the parts in $\mathcal{P}_R$ lie in an annulus (bounded away from $0$), the bound follows.  
\end{proof}

Bounded distortion leads to the following estimate for Birkhoff sums
\begin{lemma}
For any two  $x, x_0 \in \Sigma_R$  and $b \in \Gamma^M_x \cap \Gamma^M_{x_0}$ we have
$$
S_M\tau(b \Vert x) = S_M\tau(b \Vert x_0)(1+O(1)).
$$
\end{lemma}
\begin{proof}
For $0 \leq k \leq M$ both $\pi(\sigma^k(b \Vert x))$ and $\pi(\sigma^k(b \Vert x_0))$ are in the same cylinder $C_{\sigma^k(b)}$ which has diameter at most $\frac{1}{2^k}$.  

The Mean Value Theorem combined with bounded distortion yields
\begin{align*}
|S_M(\tau(b \Vert x)) - S_M(\tau(b \Vert x_0))| &\leq \sum_{k=0}^M| \tau(\sigma^k(b \Vert x) ) - \tau(\sigma^k(b\Vert x_0)) | \\
&\leq \sum_{k=0}^M \log |f' (\pi(\sigma^k(b \Vert x)))| - \log|f' (\pi(\sigma^k(b \Vert x_0)))| \\
&\leq \sum_{k=0}^M C |\pi(\sigma^k(b \Vert x)) - \pi(\sigma^k(b \Vert x_0 ))| \\
&\leq C \sum_{k=0}^\infty \frac{1}{2^k} < C'.
\end{align*}
\end{proof}

Returning to our estimation of the measure $\mu_b$, we pick an arbitrary $a_0 \in \adm{b}{x}{R}$ and the lemma gives
\begin{align*}
\mu_b \leq C e^{z S_M \tau(b \Vert a_0 \Vert x)}  \sum_{a \in \adm{b}{x}{R}} e^{z S_R\tau(a \Vert x)} \delta_{\pi_q(g_a)}.
\end{align*}

\subsubsection{Expressing $\mu_b$ as a convolution}\

Now for a divisor $L$ of $R$, break each $a \in \adm{b}{x}{R}$ into subwords of length $L$, i.e. write each $a=a_1 \Vert \cdots \Vert a_r$ where $l(a_i)=L$.  Then
$$
\sum_{a \in \adm{b}{x}{R}} e^{z S_R \tau(a\Vert x)} \delta_{\pi_q(g_a)} = \sum_{a \in \adm{b}{x}{R}} e^{z S_L \tau(a_1 \Vert a_2 \Vert \cdots \Vert a_r \Vert x)} e^{z S_L \tau(a_2 \Vert a_3 \Vert \cdots \Vert a_r \Vert x)} \cdots e^{z S_L \tau (a_r \Vert x)} \delta_{\pi_q(g_a)}.
$$

For each $a_i$, we will further decompose the word as a long prefix (of length $L-4$) and a short suffix.  Write $a_i = a_i^{(L-4)}\Vert a_i^{(4)}$.  In order to separate dependence on the suffixes, we replace $ e^{z S_L \tau(a_i \Vert a_{i+1} \Vert \cdots \Vert a_r \Vert x)}$ with $e^{z S_L \tau(a_i \Vert a_{i+1}^{(L-4)} \Vert x_i)}$ where $x_i \in \Sigma_R$ is some arbitrary admissible choice based on $a_{i+1}^{(L-4)}$.  Since we will be replacing many of the weights in $\mu_b$, we need to sharpen the estimate from the previous lemma.  In particular, since $\pi(\sigma^k(a_i \Vert a_{i+1}^{(L-4)} \Vert x_i))$ and $\pi(\sigma^k(a_i \Vert a_{i+1} \Vert \cdots \Vert x)$ are in the same cylinder $C_{\sigma^k(a_i \Vert a_{i+1}^{(L-4)})}$ for $0 \leq k \leq 2L-4$ we have 
$$
| S_L \tau(a_i \Vert a_{i+1} \Vert \cdots \Vert a_r \Vert x) - S_L \tau(a_i \Vert a_{i+1}^{(L-4)} \Vert x_i)| < \frac{C}{2^{L-4}}.
$$
Hence making the substitution for each of the $r-1$ subwords (no substitution is necessary for $a_R$) gives
\begin{align*}
\mu_b &\leq Ce^{z (r-1)2^{-L}}\sum_{a \in \adm{b}{x}{R}} \left[ \prod_{i=1}^r e^{z S_L \tau(a_i \Vert a_{i+1}^{(L-4)} \Vert x_i)}\right]  \delta_{\pi_q(g_a)} \\
&= Ce^{z (r-1)2^{-L}}\sum_{a \in \adm{b}{x}{R}} \left[ \prod_{i=1}^r  e^{z S_L \tau(a_i \Vert a_{i+1}^{(L-4)} \Vert x_i)} \right] [\delta_{\pi_q(g_{a_1})}*\delta_{\pi_q(g_{a_2})}*\cdots*\delta_{\pi_q(g_{a_r})}].
\end{align*}

Instead of decomposing $a$ into subwords of length $L$, we would like to start with subwords and determine which concatenations are admissible.  We may choose $a  \in \adm{b}{x}{R}$ in the following way: 
\begin{enumerate}
\item select $a_1^{(L-4)} \in \dm{b}{L-4}$ and $a_i^{(L-4)} \in \Gamma^{L-4}$ for the remaining $i=2,\ldots,r$.
\item select $a_r^{(4)} \in \adm{{a_{r}^{(L-4)}}}{x}{4}$ and  $a_i^{(4)} \in \adm{{a_{i-1}^{(L-4)}}}{{a_i^{(L-4)}}}{4}$ for the remaining $i=1,\ldots,r-1$.
\end{enumerate}
The effect is to separate the sum into an outer sum depending on the prefixes of length $L-4$ and an inner sum of the suffixes of length $4$.  We will also distribute the product $ \left[ \prod_{i=1}^r   e^{z S_L \tau(a_i \Vert a_{i+1}^{(L-4)} \Vert x_i)} \right]$ into the convolution:
\begin{align*}
\sum_{a_1^{(L-4)}, \ldots, a_r^{(L-4)}} \sum_{a_1^{(4)}, \ldots, a_r^{(4)}} \left[ e^{z S_L \tau(a_1 \Vert a_2^{(L-4)} \Vert x_1)} \delta_{\pi_q(g_{a_1})} \right] * \cdots *  \left[ e^{z S_L \tau(a_r \Vert x)} \delta_{\pi_q(g_{a_r})} \right] \\
=\sum_{a_1^{(L-4)}, \ldots, a_r^{(L-4)}} \left[ \sum_{a_1^{(4)}}  e^{z S_L\tau(a_1 \Vert a_2^{(L-4)} \Vert x_1)} \delta_{\pi_q(g_{a_1})} \right]  * \cdots * \left[ \sum_{a_r^{(4)}} e^{z S_L \tau(a_r \Vert x)} \delta_{\pi_q(g_{a_r})} \right].
\end{align*}

Define \label{etaj}
$$\eta_j := \sum_{\alpha_j^{(4)}} e^{z S_L(a_j \Vert a_{j+1}^{(L-4)} \vert x_j)} \delta_{\pi_q(g_{a_j})}$$
 as a distribution on $SL_2(q)$.  Our first observation about the measures $\eta_j$ is that the ratio of any two coefficients is bounded. For two admissible $a_j^{(4)}, \tilde{a}_j^{(4)}$
\begin{align*}
\left|S_L \tau( a_j^{(L-4)} \Vert a_j^{(4)} \Vert a_{j+1}^{(L-4)} \Vert x_j) - S_L \tau ( a_j^{(L-4)} \Vert \tilde{a}_j^{(4)} \Vert a_{j+1}^{(L-4)} \Vert x_j)\right|  \\
\leq C+ \left| S_4 \tau (a_j^{(4)} \Vert a_{j+1}^{(L-4)}\Vert x_j) - S_4 \tau(\tilde{a}_j^{(4)} \Vert a_{j+1}^{(L-4)}\Vert x_j)\right|.
\end{align*}
The two cylinders $C_{a_j^{(4)}}, C_{\tilde{a}_j^{(4)}}$ may be disjoint.  However, the distance between them is still bounded since they lie in $S(\mathcal{X}) \subset \{ x+iy : |x|<1/2, 0<y<1/2\}$.  So an application of the Mean Value Theorem gives that the second term in the inequality is less than some universal constant.  In other words, 
$$
e^{z S_L  \tau( a_j^{(L-4)} \Vert a_j^{(4)} \Vert a_{j+1}^{(L-4)} \Vert x_j)} = e^{z S_L  \tau( a_j^{(L-4)} \Vert \tilde{a}_j^{(4)} \Vert a_{j+1}^{(L-4)} \Vert x_j)} (1+O(1))
$$
So the coefficients of the sum defining $\eta_j$ are nearly flat.  In order to establish an expansion result for $\eta_j$, we will also need the following:
\begin{lemma}
\label{gen}
For any $j$, pairs of admissible suffixes of $a_j^{(L-4)}$ of length $4$ generate all of $\slt{\Zi}$.  Specifically, for any two letters $i,j \in \mathcal{P}_R$, we have
$$
\left\langle \pi(a) \pi(\tilde{a})^{-1} : a, \tilde{a} \in \adm{i}{j}{4} \right\rangle = \slt{\Zi}
$$
where $\pi: \Sigma_R \to \slt{\Zi}$.
\end{lemma}

The lemma is proved by finding admissible expansions of the four canonical generating matrices for $SL_2(\mathbb{Z}[i])$,
$$
 T_1 = \begin{pmatrix} 1 & 1 \\ 0 & 1 \end{pmatrix}, \hspace{1em} T_i= \begin{pmatrix} 1 & i \\ 0 & 1 \end{pmatrix}, \hspace{1em} Q= \begin{pmatrix} -i & 0 \\ 0 & i \end{pmatrix}, \hspace{1em} S= \begin{pmatrix}  0 & -1 \\ 1 & 0 \end{pmatrix} ,
$$
via matrices the form
$$
\begin{pmatrix} 0 & -1 \\ 1 & 0 \end{pmatrix} \begin{pmatrix} -i & 0 \\ 0 & i \end{pmatrix}^j \begin{pmatrix} 1 & -z \\ 0 & 1 \end{pmatrix}.
$$
An automated search through all sufficiently small expressions with restricted coefficients yields the required matrix expansions.

\subsubsection{Expansion via Selberg's $3/16$ Theorem} \

We are now ready to prove the expansion theorem for the $\eta_j$'s defined on page~\pageref{etaj}.  For each square-free $q$, we have the product representation $SL_2(q) \cong \prod_{p|q} SL_2(p)$ which gives rise to the following decomposition for functions defined on $SL_2(q)$:
$$L_2(SL_2(q)) = \bigoplus_{q' | q} E_{q'}$$ 
where 
\begin{align*}
E_{q'} := \{& \varphi: SL_2(q) \to \mathbb{C} | \\
 &\varphi(g) = \varphi(g')\text{ if } g=g'(q'), \\
& \langle \varphi, \psi \rangle = 0 \text{ for all } \psi \in E_{q''} \text{ such that } q'' | q' \text{ and }  q''<q'\}
\end{align*}
We first treat one $E_q$ at a time, and then assemble them using Fourier-Walsh decomposition (see Section~\ref{fw} on page~\pageref{fw}).
\begin{theorem}
\label{flat}
If $\varphi \in E_q$, then
$$|| \eta_j * \varphi||_2 \leq (1-C)||\eta_j||_1 ||\varphi||_2$$.
\end{theorem}
\begin{proof}
First, we retrace the standard steps to rewrite $||\eta_j * \varphi||_2^2$ in terms of a convolution operator.  By definition,
\begin{align*}
 || \eta_j * \varphi||_2^2  &= \langle \eta_j * \varphi, \eta_j * \varphi \rangle \\
&= \left\langle \sum_{a_j^{(4)}} \beta_{a_j^{(4)}} \delta_{\pi_q(a_j)} * \varphi, \sum_{a_j^{(4)}} \beta_{a_j^{(4)}} \delta_{\pi_q(a_j)} * \varphi \right\rangle 
\end{align*}
where we have 
$$
 \beta_{a_j^{(4)}} := e^{z S_L \tau (a_j \Vert a_{j+1}^{(L-4)}\Vert x_j)}.
$$
Expanding the square gives
\begin{align*}
 || \eta_j * \varphi||_2^2 &= \sum_{k \in G} \left[ \sum_{a_j^{(4)}} \beta_{a_j^{(4)}} \delta_{\pi_q(a_j)} * \varphi (k) \right]^2 \\
&= \sum_k \sum_{a_j^{(4)}, \tilde{a}_j^{(4)}} \beta_{a_j^{(4)}} \beta_{\tilde{a}_j^{(4)}} \varphi(k \pi_q(a_j)^{-1}) \varphi (k \pi_q(\tilde{a}_{j})^{-1}).
\end{align*}
(Recall that $a_j = a_j^{(L-4)} \Vert a_j^{(4)}$ and $\tilde{a}_j = a_j^{(L-4)} \Vert \tilde{a}_j^{(4)}$.)  We reorder the following sums as
\begin{align*}
 || \eta_j * \varphi||_2^2 &= \sum_k \sum_{a_j^{(4)}, \tilde{a}_j^{(4)}} \beta_{a_j^{(4)}}  \beta_{\tilde{a}_j^{(4)}}  \varphi(k \pi_q(a_j^{(4)})^{-1}) \varphi (k \pi_q(\tilde{a}_j^{(4)})^{-1})  \\
&= \sum_k \sum_{a_j^{(4)}, \tilde{a}_j^{(4)}} \beta_{a_j^{(4)}}  \beta_{\tilde{a}_j^{(4)}} \varphi(k \pi_q(a_j^{(4)})^{-1} \pi_q(\tilde{a}_j^{(4)} )) \varphi (k).
\end{align*}
Now, from Lemma~\ref{gen} and the analogue of Selberg's $3/16$ theorem for congruence subgroups of $\slt{\Zi}$ (see \cite{sarnak} or Theorem 6.1 in \cite{EGM}) we deduce that for any $\varphi \in l_0^2(\slt{q})$ there is some choice of $a_0, \tilde{a}_0 \in a_j^{(4)}, \tilde{a}_j^{(4)}$ so that
$$
|| \pi_q(a_0) \pi_q(\tilde{a}_0)^{-1} * \varphi - \varphi||_2>\epsilon ||\varphi||_2.
$$
The law of cosines gives
$$
\left\Vert \pi_q(a_0) \pi_q(\tilde{a}_0)^{-1}* \varphi - \varphi\right\Vert^2 = 2 - 2  \left\langle \pi_q(a_0) \pi_q(\tilde{a}_0)^{-1}* \varphi, \varphi \right\rangle
$$
and so
$$
\beta_{a_0} \beta_{\tilde{a}_0} \left\langle   \pi_q(a_0) \pi_q(\tilde{a}_0)^{-1}* \varphi, \varphi \right\rangle < \beta_{a_0} \beta_{\tilde{a}_0} (1-\epsilon') ||\varphi||_2.
$$
We separate the $a_0, \tilde{a}_0$ term from the rest of the sum as follows
$$
||\eta_j * \varphi||_2^2 \leq \sum_{a_j^{(4)}, \tilde{a}_j^{(4)}}' \beta_{a_j^{(4)}}  \beta_{\tilde{a}_j^{(4)}} ||\varphi||_2 +  \beta_{a_0} \beta_{\tilde{a}_0} (1-\epsilon') ||\varphi||_2
$$
Since we established earlier that $\beta_{\alpha_j}= \beta_{\alpha_j'}(1+O(1))$,  this gives
$$||\eta_j * \varphi ||_2^2 \leq (1-C) ||\eta_j||_1^2 ||\varphi||_2^2.$$
\end{proof}

Apply Theorem~\ref{flat} to each $\eta_j$ and we have
\begin{corollary}
\label{expcor}
 For $\varphi \in E_q$
\begin{align*}
\left\Vert \sum_{a \in \adm{b}{x}{R}} \left[ e^{z S_L \tau(a_1 \Vert a_2^{(L-4)} \Vert x_1)} \delta_{\pi_q(g_{a_1})} \right] * \cdots *  \left[ e^{z S_L \tau(a_r \Vert x)} \delta_{\pi_q(g_{a_r})} \right] * \varphi\right\Vert_2 \\
 \leq (1-C)^r \left( \sum_{a \in \adm{b}{x}{R}} \prod_{i=1}^r | e^{z S_L \tau(a_i \Vert a_{i+1}^{(L-4)} \Vert x_i)}| \right) ||\varphi||_2.
\end{align*}
\end{corollary}
Next, we exploit quasi-randomness of $SL_2(q)$ to get a bound for $\mu_b$.
\begin{theorem}
\label{mu}
For $R \asymp \log \GN{q}$ with $\mu_b$ as defined previously on page~\pageref{mub}, we have
$$||\mu_b * \varphi||_2 \leq C\GN{q}^{-1/4} ||\mu_b||_1 ||\varphi||_2$$
for any $\varphi \in E_q$.
\end{theorem}

\begin{proof}
Recall
$$
\mu_b := \sum_{a \in \adm{b}{x}{R}} e^{z S_N\tau(b\Vert a \Vert x)} \delta_{\pi_q(g_a)} 
$$
Define
$$
\nu:= e^{\Re{z} S_M \tau(b \Vert a_0 \Vert x)}  \sum_{a \in \adm{b}{x}{R}} e^{\Re{z} S_R\tau(a \Vert x)} \delta_{\pi_q(g_a)}
$$
We established earlier that $|\mu_b| \leq C \nu$.  Corollary~\ref{expcor} and bounded distortion yield the following bound for $\nu$:
\begin{align*}
\left\Vert  \sum_{a \in \adm{b}{x}{R}} e^{\Re{z} S_R\tau(a \Vert x)} \delta_{\pi_q(g_a)} * \varphi\right\Vert_2  \leq (1-C)^r \left( \sum_{a \in \adm{b}{x}{R}}  e^{\Re{z} S_R\tau(a \Vert x)}  \right) ||\varphi||_2
\end{align*}
and therefore $||\nu * \varphi||_2 \leq (1-C)^r ||\nu||_1 ||\varphi||_2$.

Define $A$ as the convolution operator $\varphi \mapsto \mu_b * \varphi$.  First note that $A$ acts on $E_q$ since it's a linear combination of convolutions with delta functions.  Since $A^* A $ is self adjoint, we have $\mtr{A^*A} = \lambda_1^2 + \ldots  \lambda_{|G|-1}^2$.
\begin{align*}
\lambda^2 \text{mult}(\lambda) &\leq \mtr {A^* A} \\
&= \sum_{g \in G} \langle (A^* A)^2 \delta_g, \delta_g \rangle \\
& = \sum_{g \in G} ||\tilde{\mu}_b * \mu_b * \delta_g ||_2^2\\
 &= |G| ||\tilde{\mu}_b * \mu_b ||_2^2  \\
&\leq C |G| ||\tilde{\nu} * \nu||_2^2.
\end{align*}
where $\tilde{\mu}_b(g) = \overline{\mu_b}(g^{-1})$ and a similar definition applies to $\nu$.  The multiplicity of $\lambda$ equal to the dimension of the eigenspace is at least $\frac{\GN{q}-1}{2}$ by the Fobenius lemma.  So 
$$||A||_{op} = \max \lambda^{1/2} \leq C \left( \frac{|G| || \tilde{\nu} * \nu||_2^2}{\GN{q}} \right)^{1/4}.$$
We bound $||\tilde{\nu} * \nu||_2$ by introducing $\psi = \delta_e - \frac{1}{|G|} {\bf 1}_G$.  Observe $\psi \in l_0^2$ and $||\psi||_2<1$.  
\begin{align*}
||\tilde{\nu}*\nu||_2 &= ||\tilde{\nu}*\nu*\delta_e||_2 \\
& \leq ||\tilde{\nu}*\nu* \frac{{\bf 1}_G}{|G|}||_2 + ||\tilde{\nu}*\nu*\psi||_2 \\
& \leq \frac{||\nu||_1^2}{|G|^{1/2}} + ||\nu||_1||\nu * \psi||_2.
\end{align*}
Since $||\nu*\psi||_2 < (1-C)^R ||\nu||_1$, we can choose $R = C' \log q$ to get $||\tilde{\nu}*\nu||_2<\frac{||\nu||_1}{G^{1/2}}.$
\end{proof}

\subsubsection{Applying Theorem~\ref{mu}}\

Now we would like to use the previous bound on our congruence transfer operators defined on page~\pageref{to}.  Rewrite $M^n_z$ as
\begin{align*}
M_s^n f(x,g) &= \sum_{a \in \ad{x}{n}} e^{s S_n \tau( a \Vert x)} f(a \Vert x; g_a g)  \\
&= \sum_{a \in \ad{x}{M}}  \sum_{b \in \ad{a}{R}} e^{sS_n\tau(b \Vert a \Vert x)} f(b \Vert a \Vert x; g_b g_a g) \\
&= \sum_{a \in \ad{x}{M}}  \sum_{b \in \ad{a}{R}} e^{sS_n\tau(b \Vert a \Vert x)} f(b \Vert x_b; g_b g_a g) + O\left(|f|_\rho \rho^M \sum_{a \in \Gamma_x^n} e^{s S_n\tau(a\Vert x)} \right).
\end{align*}
where $x_b$ is arbitrarily chosen as long as $b \Vert x_b$ is admissible.  We will frequently use the fact that for $s \leq \delta_R$ and the transfer operators defined in the non-congruence setting we have
$$
|\mathcal{L}_s^n 1(x)| = \left| \sum_{a \in \Gamma_x^n} e^{z \tau(a\Vert x)}\right| < \lambda_{\Re{s}}^n.
$$
Filling this in for the second term in our bound for the congruence transfer operator gives
\begin{align*}
M_z^n f(x,g) &= \sum_{b \in \Gamma^R}  \sum_{a \in \adm{b}{x}{M}} e^{\tau^n(b \Vert a \Vert x)} f(b \Vert x_b; g_b g_a g) + O\left(|f|_\rho \rho^M \lambda_{\Re{s}}^n \right) \\
&= \sum_{b \in \Gamma^R}  [\mu_b *  f(b \Vert x_b; g_b \cdot) ](g) +  O\left(|f|_\rho \rho^M \lambda_{\Re{s}}^n \right).
\end{align*}
If $f(b || x_b, g_b \cdot) \in E_q$, then we are in position to use the bound for $\mu_b$:
$$
||\mu_b * f(b\Vert x_b ; g_b \cdot) ||_2 < C \GN{q}^{-1/4} ||\mu_b||_1 ||f(b\Vert x_b ; g_b \cdot) ||_{l_2(G)}< C\GN{q}^{1/4} ||\mu_b||_1 ||f||_\rho.
$$
For any $b \in \Gamma^R$,
\begin{align*}
\Vert \mu_b\Vert_1 &= \left\Vert \sum_{a \in \adm{b}{x}{R}} e^{s S_n \tau(b \Vert a \Vert x)} \delta_{g_a} \right\Vert_1 \\
&\leq \sum_{a \in \adm{b}{x}{R}} \left| e^{s S_R \tau(b \Vert a \Vert x)} e^{s S_M \tau(a \Vert x)}\right| \\
&\leq C |e^{s S_R \tau(b \Vert a_0 \Vert x)}| \sum_{a \in \adm{b}{x}{R}} | e^{s S_M \tau(a \Vert x)}| \\
&\leq C e^{\Re{s}  S_R \tau(b \Vert a_0 \Vert x)} \mathcal{L}_{\Re{s}}^R 1(x).
\end{align*}
Apply this bound for each summand in $M_s^n$,
\begin{align*}
||M_s^n f ||_\infty &\leq C \GN{q}^{-1/4} \sum_{b \in \Gamma^R}e^{\Re{s}  S_R \tau(b \Vert a_0 \Vert x)}  \left\Vert\mathcal{L}_{\Re{s}}^R 1(x) \right\Vert_\infty||f||_\infty + O(|f|_\rho \rho^M \lambda_{\Re{s}}^n)  \\
&\leq C \GN{q}^{-1/4} \Vert \mathcal{L}_{\Re{s}}^n 1(x) \Vert_\infty \Vert f \Vert_\infty+ O(|f|_\rho \rho^M \lambda_{\Re{s}}^n) \\
&\leq C \GN{q}^{-1/4} \lambda_{\Re{s}}^n \Vert f \Vert_\infty+ O(|f|_\rho \rho^M \lambda_{\Re{s}}^n).
\end{align*}

\subsection{Bounding $M_z$ in Variation}\
Now, we need to bound $|M_z^n f|_\rho$.   Suppose $x, y \in \Sigma$ and $w \in \ad{x}{k} \cap \ad{y}{k}$,
\begin{align*}
|M_z^n &f(w \Vert x; g) - M_z^n f(w\Vert y; g)| \\ & =\left| \sum_{a \in \ad{w}{n}} e^{z S_n \tau(a \Vert w \Vert x)} f(a\Vert w\Vert x, g_a g) - e^{z S_n \tau(a \Vert w \Vert y)} f(a \Vert w\Vert y, g_a g)\right| \\
&\leq \left| \sum_{a\in \ad{w}{n}} e^{z S_n \tau( a \Vert w \Vert x)} (f(a \Vert w \Vert x; g_a g) - f(a \Vert w \Vert y; g_a g))\right|  \\
&+\left| \sum_{a\in\ad{w}{n}} ( e^{z S_n \tau( a \Vert w \Vert x)} - e^{z S_n \tau( a \Vert w \Vert y)}) f(a\Vert w \Vert y, g_a g) \right|.
\end{align*}
For the first term, we note that $a \Vert w \Vert x$ and $a \Vert w \Vert y$ agree in the first $n+k$ letters so
\begin{align*}
 \left| \sum_{a\in \ad{w}{n}} e^{z S_n \tau( a \Vert w \Vert x)} (f(a \Vert w \Vert x; g_a g) - f(a \Vert w \Vert y; g_a g))\right| 
&\leq |f|_\rho \rho^{n+k} \sum_{a \in \ad{w}{n}}| e^{z S_n \tau( a \Vert w \Vert x)}| \\
&\leq |f|_\rho \rho^{n+k} \mathcal{L}_{\Re{z}}^n 1 (w \Vert x) \\
&\leq|f|_\rho \rho^{n+k} \lambda_{\Re{z}}^n.
\end{align*}
For the second term, we will use a similar approach as before. Decompose $\Gamma^n$ into $\Gamma^R \otimes \Gamma^M$ and decouple:
\begin{align*}
 \sum_{a\in\ad{w}{n}} &( e^{z S_n \tau( a \Vert w \Vert x)} -  e^{z S_n \tau( a \Vert w \Vert y)} ) f(a\Vert w \Vert y, g_a g)  \\
&= \sum_{a \in \ad{w}{M}} \sum_{b \in \ad{a}{R}} ( e^{z S_n \tau( b \Vert a \Vert w \Vert x)} -  e^{z S_n \tau(  b \Vert a \Vert w \Vert y)} ) f(b\Vert a \Vert w \Vert y ; g_b g_a g)  \\
&=\sum_{a \in \ad{w}{M}} \sum_{b \in \ad{a}{R}} ( e^{z S_n \tau( b \Vert a \Vert w \Vert x)} -  e^{z S_n \tau(  b \Vert a \Vert w \Vert y)} )  f(b\Vert x_b ; g_b g_a g)  \\
&{\color{white}.}\hspace{.8cm}+ O\left(|f|_\rho \rho^M \sum_{a \in \Gamma_x^n}  ( e^{z S_n \tau( a \Vert w \Vert x)} -  e^{z S_n \tau( a \Vert w \Vert y)} ) \right).
\end{align*}
For the error term, we estimate
\begin{align*}
 \sum_{a \in \Gamma_x^n}  ( e^{z S_n \tau( a \Vert w \Vert x)} -  e^{z S_n \tau( a \Vert w \Vert y)} )  
\leq  \sum_{a \in \Gamma_x^n} |e^{z S_n \tau(a \Vert w \Vert x)}| | 1 - e^{z (S_n \tau( a \Vert w \Vert x) - S_n ( a \Vert w \Vert y))}|.
\end{align*}
For the first term, we use the eigenvalue bound.  For the second, we use the fact that $\pi(a \Vert w \Vert x)$ and $\pi(a \Vert w \Vert y)$ are in the same $n+k$ cylinder combined with bounded distortion (i.e. $ |(S_n \tau( a \Vert w \Vert x) - S_n ( a \Vert w \Vert y))| < C 2^{-k}$.)  This gives
\begin{align*}
 \sum_{a \in \Gamma_x^n}  ( e^{z S_n \tau( a \Vert w \Vert x)} -  e^{z S_n \tau( a \Vert w \Vert y)} )  < C \lambda_{\Re{z}}^n (1+|\Im{z}|) 2^{-k}.
\end{align*}

\subsubsection{Applying Theorem~\ref{mu} Again}\

For each $b \in \Gamma_R$ define
$$
\mu_b := \sum_{a \in \adm{b}{w}{M}} (e^{\tau^n(b\Vert a \Vert w \Vert x)} - e^{\tau^n(b \Vert a \Vert w \Vert y)}) \delta_{\pi_q(g_a)}.
$$
The same proof as before follows through for $\mu_b$, as long as $R \asymp \log \GN{q}$, to yield the following for $\varphi \in E_q$:
$$
||\mu_b* \varphi||_2 < C\GN{q}^{-1/4} ||\mu_b||_1 ||\varphi||_2.
$$
Referring to the proof, we note that the $\mu_b$'s have slightly different coefficients in the corresponding $\eta_j$'s: 
$$
\eta_j := \sum_{a_j^{(4)}} (e^{z S_L (a_j \Vert a_{j+1}^{(L-4)} \Vert x_j)} - e^{z S_L (a_j \Vert a_{j+1}^{(L-4)} \Vert x_j)}) \delta_{\pi_q(a_j)}.
$$
However, the important property of `nearly flat coefficients' (i.e. that the constant for each summand varies by at most a constant ratio) is preserved.  Returning to the bound for $M_s^n$, we have
\begin{align*}
|M_s^n f |_\rho &\leq C \GN{q}^{-1/4} \sum_{ b \in \Gamma^R} ||\mu_b||_1 ||f||_\infty  + O(|f|_\rho \rho^M( \lambda_{\Re{s}}^n(1+|\Im z|)2^{-k} + \lambda_{\Re{s}}^{n+k})).
\end{align*}
For each $\mu_b$ we have
\begin{align*}
||\mu_b||_1 &\leq \sum_{a \in \Gamma_x^M} |e^{z S_n \tau(a \Vert w \Vert x)}| \cdot | 1 - e^{z (S_n \tau( a \Vert w \Vert x) - S_n ( a \Vert w \Vert y))}| \\ &< C \lambda_{\Re{z}}^M(1+|\Im{z}|) 2^{-k} |e^{z S_R(b \Vert x_b)}|.
\end{align*}

So finally we have
\begin{align*}
|M_s^n f|_\rho \leq &C\GN{q}^{-1/4} \lambda_{\Re{s}}^n ||f||_\infty (1+|\Im{z}|) 2^{-k} \\ &+ |f|_\rho \rho^M \lambda_{\Re{s}}^n(1+|\Im z|)2^{-k} + |f|_\rho \rho^M\lambda_{\Re{s}}^{n+k}.
\end{align*}
Recall also our bound for $||M_s^n f||_\infty$:
$$
||M_s^n f||_\infty \leq C \GN{q}^{-1/4} \lambda_{\Re{s}}^n \Vert f \Vert_\infty+ |f|_\rho \rho^R \lambda_{\Re{s}}^n.
$$
We assume $n>\log \GN{q}$ because we needed $R=c\log \GN{q}$ and $n=M+R$.  Thus, we have
\begin{align*}
|M_s^n f|_\rho &\leq C \lambda_{\Re{s}}^n \GN{q}^{-1/4} (1+|\Im{z}|) (|| f ||_\infty + \rho^{n/2} |f|_\rho) \\
||M_s^n f ||_\infty &\leq C \lambda_{\Re{s}}^n \GN{q}^{-1/4}(|| f ||_\infty + \rho^{n/2} |f|_\rho).
\end{align*}
Further, take $n \sim \log q + C \log(1 + |\Im{z}|)$ and we have
$$
||M_s^n f ||_\infty + \rho^{n/2} | M_s^n f|_\rho < \lambda_{\Re{s}}^n \GN{q}^{-1/4} ( ||f ||_\infty + \rho^{n/2} |f|_\rho).
$$
Iterating the inequality yields
$$
||M_s^{mn} f ||_\rho < \lambda_{\Re{s}}^{mn} \GN{q}^{-m/4} \GN{q}(1+|\Im{s}|) ||f||_\rho.
$$

\subsection{Fourier-Walsh Decomposition}\
\label{fw}

We would like to extrapolate from the previous bound (valid for $f \in E_q \subset L^2(\slt{q})$) a bound for any of the non-constant level subspaces.  In particular, recall that $L_2(SL_2(q)) = \bigoplus_{q' | q} E_{q'}$.  We temporarily denote $E_{q'}$ inside of $L_2(SL_2(q))$ as $E_{q'}(q)$ in order to compare $E_{q'}(q)$ with $E_{q'}(q')$.  This decomposition extends to one for $C(\Sigma_R \times \slt{q})$, namely
$$
C(\Sigma_R \times \slt{q}) = \bigoplus_{q'|q} C(\Sigma_R, E_{q'}).
$$

\begin{enumerate}

\item $M_{z,q}$ preserves the subspaces $C(\Sigma_R, E_{q'}) $ because 
$$
M_{z,q} f(x, \cdot) = \sum_{a \in \Gamma_x^1} e^{z \tau(a \Vert x)} f(a \Vert x, g_a \cdot)
$$
and $g \mapsto g_a g$ is an automorphism of $\slt{q}$ for each $a$.  So the right hand side is a linear combination of functions in  $C(\Sigma_R, E_{q'}) $.

\item The natural projection from $\slt{q}$ to $\slt{q'}$ extends to the subspaces $E_{q'}(q)$ and $E_{q'}(q')$.   In particular, $f \in E_{q'}(q)$ and the corresponding $\hat{f} \in E_{q'}(q')$ satisfy
$$
||f||_{L^2(q)} = \sqrt{\frac{|SL_2(q)|}{|SL_2(q')|}} ||\hat{f}||_{L^2(q')}
$$
and if we denote the norm $|| \cdot ||_{\rho, q'}$ on $\mathcal{F}_\rho( \Sigma_R, E_q')$ we have
$$
||f||_{\rho, q} \leq \sum_{q' | q} \sqrt{\frac{|SL_2(q)|}{|SL_2(q')|}} ||\hat{f}||_{\rho, q'}.
$$

\item  $\widehat{M_{z,q} (1 \otimes f)} = M_{z,q'} (1 \otimes \hat{f})$, i.e. the projection is equivariant under the appropriate congruence transfer operators.

\end{enumerate}
These three properties allow us to decompose a function $f = \sum_{q'|q} f_{q'}$ where $f_{q'} \in C(\Sigma_R \times E_{q'})$ and apply our bound as if $f_{q'}$ is in $\mathcal{F}_\rho(\Sigma, E_{q'}(q'))$.  Assume $f_1=0$, i.e. $f$ is orthogonal to the constant function. 

\subsubsection{Small Imaginary Part}\
For small imaginary part ($|\Im{s}|<\GN{q}$)
\begin{align*}
 \sum_{1\neq q' | q} &\sqrt{\frac{|SL_2(q)|}{|SL_2(q')|}} ||M^m_{z,q'} \hat{f_{q'}}||_{\rho,q'} \\
&\leq \sum_{1\neq q' | q} \sqrt{\frac{|SL_2(q)|}{|SL_2(q')|}} \lambda_{\Re{s}}^{m} \GN{q}^{-m/(4n)} \GN{q'}(1+|\Im{s}|) ||\hat{f_{q'}}||_{q'}  \\
&\leq  \sum_{1\neq q' | q} \sqrt{\frac{|SL_2(q)|}{|SL_2(q')|}} \GN{q'}^C e^{-\epsilon n} \lambda_{\Re{z}}^n ||\hat{f_{q'}}||_{q'} \\
&\leq \GN{q}^C e^{-\epsilon n} \lambda_{\Re{z}}^n ||f||_q
\end{align*}
where we used that $||\hat{f_{q'}}||_{q'} \leq ||f||_q$ and that the number of divisors of $q$ is at most $\GN{q}^{\epsilon'}$.

\subsubsection{Large Imaginary Part}\
For large imaginary part $(|\Im{s}|>\GN{q})$, we have
\begin{align*}
\sum_{1 \neq q' | q}& \sqrt{\frac{|SL_2(q)|}{|SL_2(q')|}} ||M^m_{z,q'} \hat{f_{q'}}||_{\rho,q'} \\
&\leq \sum_{1 \neq q' | q } \sqrt{\frac{|SL_2(q)|}{|SL_2(q')|}}\lambda_{\Re{s}}^{m} \GN{q}^{-m/(4n)} \GN{q'}(1+|\Im{s}|)  ||\hat{f}_{q'}||_{q'}  \\
&\leq \sum_{1 \neq q' | q } \sqrt{\frac{|SL_2(q)|}{|SL_2(q')|}}  |\Im{z}|^C e^{-\epsilon n \frac{\log \GN{q'}}{\log |\Im{z}|}} \lambda_{\Re{z}}^n ||\hat{f}_{q'}||_{q'} \\
&\leq ||f||_q |\Im{z}|^C \lambda_{\Re{z}}^n \sum_{1 < \GN{q'}<|\Im{z}|} e^{-\epsilon n \frac{\log \GN{q'}}{\log |\Im{z}|}}.
\end{align*}
To bound the sum, observe
\begin{align*}
\sum_{1 < \GN{q'}<|\Im{z}|} e^{-\epsilon n \frac{\log \GN{q}}{\log |\Im{z}|}} 
&< \prod_{p|q} \left(1 + \exp\left(\epsilon n \frac{-\log \GN{p}}{\log|\Im{z}|} \right)\right) -1\\
&\leq  \exp\left( \sum_{p} e^{-\epsilon n \frac{\log \GN{p}}{\log |\Im{z}|} }\right) - 1 \\
&< C \exp\left( \sum_{s=2}^\infty e^{-\epsilon n \frac{\log s}{\log |\Im{z}|} }\right) - 1 \\
&< e^{-\frac{\epsilon n}{\log |\Im{z}|}}.
\end{align*}

We have shown the following:
\begin{theorem}
\label{Mbound}
For $f$ orthogonal to the constant functions, there is some $\epsilon$ so that
$$
||M_s^n f||_\rho \leq\begin{cases}
\GN{q}^C e^{-\epsilon n} \lambda_{\Re{s}}^n||f||_\rho & |\Im{z}| \leq \GN{q} \\
|\Im{s}|^C e^{\frac{-\epsilon n}{\log | \Im{s}|} } \lambda_{\Re{s}}^n ||f||_\rho & |\Im{z}| \gg \GN{ q}
\end{cases}.
$$
\end{theorem}

Now, in order to find the region of analyticity of $(1-M_z)^{-1}$, we recall that $z \mapsto \lambda_z$ is holomorphic in a small neighborhood of $\delta_R$ and $\lambda_{\delta_R} = 1$.  In particular, this means there is some $\epsilon_2$ such that $\lambda_{\Re{z}}<e^{\epsilon/2}$ for all $\Re{z} \in [-\delta, -\delta + \epsilon_2]$.  We can also find $\epsilon_3$ so that for $z$ satisfying $\Re{z} \in [-\delta, -\delta+ \epsilon_3 \frac{1}{\log |\Im{z}|}]$ we have
$$
\lambda_{\Re{z}} < e^{\epsilon/2 \frac{1}{\log|\Im{z}|}}.
$$
For $z$ in both regions, we have that $(1-M_z)^{-1}$ (restricted to the space orthogonal to constant functions) is holomorphic and bounded by
$$
||(1-M_z)^{-1}||_\rho < (\GN{q}+|\Im{z}|)^C.
$$

\subsection{Fourier Analysis of the Renewal Function}\

Similar to the previous chapter, we introduce a counting function which satisfies a functional equation relating it to the resolvent of the congruence transfer operator $M_z$ on $\mathcal{F}_\rho(\Sigma_R \times \slt{q})$.  To incorporate the congruence aspect, we define
$$
N_q(X,a,g,f) := \sum_{n=0}^\infty \sum_{b \in \Gamma_a^n} g(b\Vert a) f(\pi_q(b)) {\bf 1}_{\{S_n \tau (b \Vert a) \leq X\}}.
$$
where $g$ is a function on $\Sigma_R$, $f$ is a function on $\slt{q}$, $a \in \Sigma_R$, and $X>0$.  The renewal equation is
$$
N_q(X,a,g,f) = \sum_{b \in \Gamma_a^1} N_q(X-\tau(b\Vert a),b\Vert a, g,\rho(\pi_q(b))f) + g(a)f(1) \ind{0 \leq X}
$$
where $\rho$ denotes the right regular representation of $\slt{q}$.  In particular
$$
\rho(\pi_q(b))f (z) = f(z \pi_q(b)).
$$
The Laplace transform
$$
F_q(s,a,g,f):=\int_{-\infty}^\infty e^{-st} N_q(t,a,g,f) dt
$$
satisfies
$$
sF_q(s,a,g,f) = (1-M_{s,q})^{-1} [g \otimes f] (a) .
$$
Observe that $N_q(X,a,g,f)$ is linear in $f$ and so is its Laplace transform.  The main contribution to $F_q(s,a,g,f)$ comes from the constant term and this analysis is a straightforward extension of the previous chapter since
$$
N_q(t,a,g,{\bf 1}) = N(t,a,g) .
$$
The contribution from functions orthogonal to constants is bounded using Theorem~\ref{Mbound}.

As in \cite{BGS}, we can choose a smooth nonnegative function $k$ on $\mathbb{R}$ such that $\int k = 1$, $\text{supp}(k) \subset [1,1]$ with the following bound for its Fourier transform
$$
|\hat{k}(\xi)| \leq C \exp ( - |\xi|^{1/2}).
$$
For some small parameter $\gamma$, we define
$$
k_\gamma(t) = \frac{1}{\gamma} k\left(\frac{t}{\gamma} \right).
$$
Note that $|\hat{k}_\gamma(\xi)| \leq C \exp(-|\gamma \xi|^{1/2})$.

Inserting the smoothing function gives
$$
\int_{-\infty}^\infty k_\gamma (t) N_q(X+t, a, f) dt = \frac{1}{2 \pi i} \int_{\delta_R + i \mathbb{R}} e^{Xs} \hat{k}_\gamma(s)  (I-M_s)^{-1} ds.
$$

Theorem~\ref{Mbound} allows us to shift the contour by
$$
-\delta_R+it \mapsto -\delta_R + \frac{\epsilon}{2} \min\left(1, \frac{1}{\log(1+|t|)}\right) + it
$$
and gives
\begin{theorem}
\label{eqbound}
For $\varphi \in L^2(\slt{q})$ with $\langle \varphi, 1 \rangle = 0$,
$$\left| \int_{-\gamma/2}^{\gamma/2} k_\gamma (t) N(T+t, \varphi, a) dt \right| < \GN{q}^{C} \gamma^{-C}  \exp\left(\min\left(1,{\frac{1}{\log \frac{T}{\gamma}}}\right)\right)e^{T\delta} ||\varphi||_2.$$
\end{theorem}

\subsection{The Finite Renewal Function}\

As in the previous chapter, the geodesic count comes from an analysis of a lattice point counting function which is close to our renewal function $N_q$ for large sequences.  The strategy is analogous to \cite{Lalley} or \cite{MOW}.

We define
$$
N_q^*(X,a,g,f):= \sum_{n=0}^\infty \sum_{b \in \Gamma_a^n} g(b\Vert a) f(\pi_q(b)) {\bf 1}_{\{S_n \tau^* (b \Vert a) \leq X\}}
$$
where $a \in \Sigma_R^* \cup \Sigma$ (see Section~\ref{finren} for the definition of the space $(\Sigma^* \cup \Sigma, \sigma)$  and $\tau*$.)   To count geodesics with congruence conditions, it suffices to provide an asymptotic for
\begin{align*}
N^*(T, \emptyset, {\bf 1}, \delta_{\pi_q(g)}) &= \#\{ b \in \Sigma_R^* : g_b \equiv g \bmod q, d(j, g_b j) < T\} \\
&= \#\{ b \in \Sigma_R^* : g_b \equiv g \bmod q, ||g_b||< \sqrt{2 \cosh T}\} .
\end{align*}

Iterating the finite renewal equation yields
\begin{align*}
N^*(X,a,g,f) = &\sum_{b \in \Gamma_a^n } N^*(X-S_n\tau^*(b), b,g,\rho(\pi_q(b))f) \\
&+\sum_{k=1}^{n-1} \sum_{b \in \Gamma_a^k} g(b\Vert a)f(\pi_q(b)) \ind{S_k \tau^*(b) \leq X} + g(a)f(1)\ind{X \geq 0}.
\end{align*}
As $X \to \infty$, the second line does not change.   For each summand corresponding to $b \in \Gamma_a^n$ in the first line we can find $b' \in \Sigma_R$ to sandwich between the two terms
$$
N_q(T \pm 2^{-n+k} -S_n \tau^*(b'), b',g,\rho(\pi_q(b))f) \asymp N^*_q(T-S_n \tau^*(b), b,g,\rho(\pi_q(b))f) .
$$
So, it suffices to analyze  $N_q(T \pm 2^{-n+k} -S_n \tau^*(b'), b',g,\rho(\pi_q(b)) \ind{\pi_q(g)})$ (sending $n \to \infty$ to get the final theorem).  Let $\varphi = \delta_{\pi_q(g)}$ and write
$$
\varphi = \frac{\langle \varphi, 1 \rangle}{|\slt{q}|} 1 + \varphi'
$$
where $\langle \varphi', 1 \rangle = 0$.   Then we have
\begin{align*}
\int_{-\infty}^\infty k_\lambda(t) N(T+t, a, \varphi) dT =&\frac{\langle \varphi, 1 \rangle}{|\slt{q}|}  \frac{1}{2\pi i} \int_{\delta+i\mathbb{R}} e^{as} F(s,a, 1) \hat{k}\lambda(s) ds \\
&+  \frac{1}{2\pi i} \int_{\delta+i\mathbb{R}} e^{as} F(s,a, \varphi') \hat{k}\lambda(s) ds.
\end{align*}
Observe $N_q(T,a,1) = N(T,a)$ where $N(T,a)$ is the renewal function defined in the previous section.  We established in the previous section that
$$
N(T-O(2^{-N}), b) \leq N^*(T,a) \leq N(T+O(2^{-N}), b)
$$
when $a$ and $b$ are in the same $N$-cylinder.   So the first term is $N^*(T,a) + O(e^{T(\delta-c)})$.   Since $N_q(T,a,g,f)$ is increasing in $T$, we have
$$
N_q(T-\gamma,a,g,f) \leq \int_{-\gamma}^\gamma k_\gamma(t) N_q(T+t,a,g,f) dt \leq N(T+\gamma,a,g,f).
$$
Sending $\gamma \to 0$ and appealing to Theorem~\ref{eqbound} gives
$$
N_q(T,a,g,\varphi') \ll  \GN{q}^{C} (e^{ -T \epsilon_3} + e^{\frac{1}{\log T}})e^{T\delta_R}.
$$
After renaming constants, we have shown Theorem~\ref{affinesieve} from the introduction:
\begin{theorem}
For each $R>8$, there is some absolute spectral gap $\Theta_R>0$ and absolute constants $c_R, C_R>0$ such that for all square-free $q \in \Zi$ and $\omega \in \slt{q}$ we have the estimate
$$
\left| \#\{g \in \Gamma_R \cap B_X : g \equiv \omega \bmod q\} - \frac{\#(\Gamma_R \cap B_X)}{|\slt{q}|} \right| \ll_R \#(\Gamma_R \cap B_X) E(q,X)
$$
as $X \to \infty$, where 
$$
E(q,X) = \begin{cases} e^{-c_R \sqrt{\log X}} & \GN{q} < C_R \log X \\ \GN{q}^{C_R} X^{-\Theta_R} & \GN{q}> C_R\log X \end{cases}.
$$
\end{theorem}

\section{The Growth Parameter $\delta_R$}
\label{delta}

Recall from Chapter~\ref{lall} that we proved the following asymptotic
$$\# \Gamma_R \cap B_X \asymp X^{2\delta_R}.$$
The exponent $\delta_R$ was the unique solution to the pressure equation
$$
P(s) =  \log \lambda_s = 0.
$$
In other words, the function $s \mapsto \lambda_s$ (where $\lambda_s$ is the maximal eigenvalue of $\mathcal{L}_s$ on $s \in \mathbb{S}$) is strictly decreasing in $s$ and $\lambda_{\delta_R}=1$.

We will make use of the fact that $\delta_R$ is arbitrarily close to $2$ as $R \to \infty$ in a later section.  Therefore, we now show that $\delta = \delta(R) \to 2$ as $R \to \infty$.  In order to proceed, we need to consider the action of the transfer operators on a different space where it becomes easier to compare $\mathcal{L}_{s,R}$ for varying $R$.

Previously, we have considered transfer operators on the subshift of finite type $(\Sigma_R, \sigma)$.  In order to show that $\delta_R \to 2$ as $R \to \infty$, we will need to compare the dynamics in $(\Sigma_R, \sigma)$ to that of the subshift on a countable alphabet $(\Sigma, \sigma)$.  

Recall the set of admissible words in the countably infinite alphabet $\mathcal{P}$ is
$$
\Sigma := \{ (p_1, p_2, \ldots) : p_i \in \mathcal{P}, f(p_i) \supset p_{i+1} \}.
$$
We work in the space of H\"{o}lder continuous functions $\mathcal{F}_\rho (\Sigma)$ with the norm
$$
||\cdot ||_\rho = ||\cdot ||_\infty + |\cdot|_\rho.
$$
Also let $\mathcal{F}_\infty(\Sigma)$ be the space of continuous functions endowed with the sup norm.  Consider the infinite transfer operator
\begin{align*}
\mathcal{L}_{s,\infty} g(x) := \sum_{y \in \Sigma: \sigma(y)=x} e^{-s \tau(y)} g(y)
\end{align*}
where the distortion function $\tau$ is the same as before: $\tau(y) = \log |f'(\pi(y))|$.  In order to establish that $\mathcal{L}_{s,\infty}$ is a bounded linear operator, we need that $s \tau$ is summable, i.e.
$$
\sum_{y \in E^1} \left( \sup_{x \in \Sigma}{e^{-s \tau(y \Vert x)}} \right) < \infty.
$$
In the interior of each one-cylinder, $f$ is defined as $f(z) = (-1)^k (1/z - \lfloor 1/z \rceil )$ and so $|f'(z)| = 1/|z|^2$.  This bounds our weights by
$$
|e^{-s\tau(y\Vert x)} | \leq |\pi(y \Vert x)|^{2 \Re{s}}.
$$
For each lattice point in $\Zi$ falling in $\mathcal{X}$, we have between two and six adjacent one-cylinders.  This gives
$$
\sum_{y \in E^1} \left( \sup_{x \in \Sigma}{e^{-s \tau(y \Vert x)}} \right) < 6\sum_{\substack{z \in \Zi\\ |z|\geq 2, \Im{z}\geq 0}}  \sup_{\substack{w\\ |w-z|<1/\sqrt{2}}} |w|^{-2\Re{s}} < C \sum_{z \in \Zi / \{0\}} |z|^{-2 \Re{s}}.
$$
So, for $s \tau$ to be summable, it suffices that $\Re{s}>1.$

In place of the irreducible and periodic properties of the finite subshifts $(\Sigma_R, \sigma)$, we must now have that $(\Sigma, \sigma)$ is finitely primitive.  In other words, there exists some $k$ and finite subset $W \subset E^k$ such that for each $i, j \in E^1$ there is some $\omega \in W$ such that $i \omega j$ is admissible.  This is clear from the proof we provided for the irreducibility of $(\Sigma_R, \sigma)$.  We divided the one-cylinders into eight regions, each containing a full square.  The image of each under $f$ contained two squares which in turn maps to the full region.  Thus, we can choose $W \subset E^2$ of size $16$.  

With these properties, we may apply the Perron-Frobenius theory for subshifts on a countable alphabet.  See \cite{mauldinurb} or \cite{mauldinurbbook} for proof of the following theorem
\begin{theorem} For the infinite transfer operator $\mathcal{L}_{s,\infty}$, as long as $\Re{s}>1$, we have
\begin{enumerate}
    \item The spectral radius of $\mathcal{L}_s$ acting on either $\mathcal{F}_\infty$ or $\mathcal{F}_\rho$ is $\lambda_{s,\infty} = e^{P(s\tau)}$
    \item $\lambda_{s,\infty}$ is a simple eigenvalue and has a corresponding eigenfunction $h_{s,\infty} \in \mathcal{F}_\rho \subset \mathcal{F}_\infty$ which is positive.
    \item The remainder of the spectrum on $\mathcal{F}_\rho$ is in a disc centered at $0$ with radius strictly smaller than $\lambda_{s,\infty}$.
\end{enumerate}
where $P(s \tau)$ denotes the pressure function:
$$
P(s\tau) := \lim_{n \to \infty} \frac{1}{n} \log \sum_{\omega \in E^n} \left( \sup_{x \in \Sigma} e^{s S_n \tau(\omega \Vert x)} \right).
$$
\end{theorem}

Note that the topological pressure $P(s\tau)$ is increasing in $s$ and there is a unique 0.   Combining \cite{sullivan} with \cite{series}, we have that this value is the critical exponent $\delta(\Gamma)$ which is equal to the Hausdorff dimension of the limit set of $\Gamma$ which is 2 (since the limit set has non-zero Lebesgue measure).  Thus, $\lambda_{2, \infty}=1$.

Information about $\lambda_{s,R}$ will follow from Keller-Liverani Perturbation Theorem (see Appendix A of \cite{pollicotturb} or \cite{kliverani} for proof.)  

We will work in the setting of the two norms $|| \cdot ||_\infty \leq ||\cdot||_\rho$ in the Banach space $\mathcal{F}_\rho(\Sigma)$.  Define
$$
|||Q|||:= \sup\{ ||Qf||_\infty : ||f||_\rho \leq 1\}
$$
and consider a family of operators $P_k$ compared to some `limit operator' $P$.  The next theorem will require the four following conditions:
\begin{enumerate}
    \item There are $C, M>0$ such that for all $k, n \in \mathbb{N}$, $$||P_k^n||_\infty \leq C M^n$$.
    \item There are $C_1, C_2, \alpha>0$ such that for all $n,k \in \mathbb{N}$, $$||P_k^n f ||_\rho \leq C_2 \alpha^n ||f||_\rho + C_2 M^n ||f||_\infty$$.
    \item If $z \in \sigma(P_n) \cap \overline{B}^c(0,\alpha)$, then $z$ is not in the residual spectrum of $P_n$.
    \item $|||P_k-P||| \to 0$ as $k \to \infty$.
\end{enumerate}
Although the full Keller-Liverani Perturbation Theorem provides more refined information about the spectrum of $P_k$ and $P$, we only need the following 
\begin{theorem}
Assume the family $\{P_k\}_{k \in \mathbb{N}}$ and $P$ satisfy conditions (1)-(4) above.  If $\lambda$ is a simple, isolated eigenvalue of $P$, then for every sufficiently large $k$, there exists a unique simple eigenvalue $\lambda_k$ of $p_k$ such that
$$
\lim_{k \to \infty} \lambda_k = \lambda.
$$
\end{theorem}

We will apply the theorem to the following family: $\mathcal{L}_{s,R} f := \mathcal{L}_{s,R} [{\bf 1}_{E_R^1} f]$.
In other words, we'll only sum over the one-cylinders in $\Sigma_R$.  Conditions (1) and (2) follow from $||L_{s,R}||_\infty \leq ||L_{s,\infty}||_\infty$ and $||L_{s,R}||_\infty \leq ||L_{s,\infty}||_\infty$.  Condition (3) is automatically satisfied by our choice of $||\cdot||_\infty$ in $\mathcal{F}_{\rho}(\Sigma)$.  In particular, the unit ball in $\mathcal{F}_{\rho}(\Sigma)$ is $||\cdot||_\infty$-compact by Ascoli's theorem (see \cite{parrypoll}.)  Therefore, we must now only establish property (4).

\subsubsection{Perturbation estimates for large alphabets}\
Assuming $\Re{s}>1$, we have
    \begin{align*}
    ||(L_{s, \infty} - L_{s,R})f(x)||_\infty &= \left\Vert \sum_{y \in E^1  \setminus E_R^1} e^{s \tau(y \Vert x)} f(y \Vert x) \right\Vert_\infty \\
    &< C ||f||_\infty \sum_{z \in \Zi, |z|>R-1} |z|^{-2\Re{s}} \\
    &< C \frac{||f||_\infty}{R}.
    \end{align*}

So the Keller-Liverani perturbation theorem implies $\lambda_{s,R} \to \lambda_{s,\infty}$ as $R \to \infty$ when $\Re{s}>1$.  Since $\lambda_{2,\infty}=1$ and $\lambda_s,R$ is analytic in $s$ for fixed $R$ we have that the solution to $\lambda_{s,R}=1$ lies in an epsilon neighborhood of $2$ for large enough $R$.

Finally, we must relate the eigenvalues of $\mathcal{L}_{s,R}$ on $(\Sigma, \sigma)$ to those of $\mathcal{L}_s$ on $(\Sigma_R, \sigma)$.  It suffices to show that
\begin{align*}
\lim_{n \to \infty} \frac{1}{n} \log \sum_{\omega \in E_R^n} \left( \sup_{x \in \Sigma} e^{s S_n \tau(\omega \Vert x)} \right) &= \lim_{n \to \infty} \frac{1}{n} \log \sum_{\omega \in E_R^n} \left( \sup_{x \in \Sigma_R} e^{s S_n \tau(\omega \Vert x)} \right) \\
(P_1) &= (P_2)
\end{align*}
since the left hand side is the log of the lead eigenvalue for $L_{s,R}$ on $(\Sigma, \sigma)$ and the right hand side is the log of the lead eigenvalue of $L_s$ on $(\Sigma_R,\sigma).$  

Combining the fact that the diameter of an $n$ cylinder is at most $2^{-n}$ with bounded distortion, we have
$$
\sup_{x \in \Sigma} e^{s S_n \tau(\omega \Vert x)} = 
\sup_{x \in \Sigma_R} e^{s S_n \tau(\omega \Vert x)} + O_R(2^{-n})
$$
where the bound is uniform over $w \in E_R^n$ for fixed $s$.  Inserting this into $(P_1)$ gives
\begin{align*}
    (P_1) &= (P_2) + \lim_{n \to \infty} \frac{1}{n} \log \left( 1 + \frac{\sum_{w \in E_R^n} \sup_{x \in \Sigma_R} e^{s S_n \tau(w \Vert x)}}{C \sum_{w \in E_R^n} 2^{-n}} \right)
\end{align*}
For $w \in E_R^n$ and $x \in \Sigma_R$, bounded distortion yields $|e^{s S_n \tau(w \Vert x)}| \in [(1/R)^{2\Re{s}}, (1/2)^{2\Re{s}}]$.  So,
$$
(P_1) = (P_2) + \lim_{n \to \infty} \frac{1}{n} \log \left(1 + O\left((\frac{2}{R})^n\right)\right) = (P_2).
$$

\section{Construction of the Multilinear Sifting Set}

Recall from Chapter~\ref{pollicott}, that the set of finite admissible words
$$
\Sigma_R^* = \{ (p_1, \ldots, p_n) : p_i \leftrightarrow \rho_i \in \mathcal{P}_R, f(\rho_i) \supset \rho_{i+1} \}
$$
represents closed geodesics in $\slt{\Zi} \backslash \mathbb{H}^3$.  Specifically, when we restrict to aperiodic words in $\Sigma_R^*$ we have an exact correspondence.  In order to detect whether a geodesic is fundamental, we will need to use the correspondence between closed geodesics and hyperbolic matrices in $\slt{\Zi}$.  We define
\begin{align*}
\pi : \Sigma_R^* &\to \slt{\Zi} \\
 a &\mapsto f|_{C_a}
\end{align*}
where $C_a$ is the cylinder $\rho_{a_{1}} \cap f^{-1}(\rho_{a_2}) \cap \cdots \cap f^{-k-1}(\rho_{a_k})$ and $f$, the Pollicott map, acts as a fractional linear transformation locally on that cylinder.  Therefore, we may express the correspondence between $\Sigma_R^*$ and closed geodesics (as primitive hyperbolic matrices in $\slt{\Zi}$) as
$$
\Gamma_R := \{ \pi (a) : a \in \Sigma_R^*, a \text{ is aperiodic} \}.
$$
In order to develop an asymptotic estimate for 
$$
\# \{ \omega \in \Gamma_R : \mtrs{\omega}-4 \text{ square-free }, ||\omega||<X\}
$$
we will use the multilinear structure coming from $\Sigma_R^*$.  

\label{siftingset}

In order to construct a large set of geodesics with the desired expansion properties, we specify the general expansion result from a previous chapter, so that we have a fixed radius $R$ and corresponding spectral gap $\Theta$:
\begin{theorem}
\label{fixtheta}
There is an absolute $\Theta>0$ (fixed throughout the remaining sections) and $c,C>0$ such that for any square-free $q \in \Zi$
$$
\#\{ \gamma \in \Gamma_8 \cap B_Y : \gamma \equiv \alpha \bmod q\} = \frac{ \# (\Gamma_8 \cap B_Y)}{|\slt{q}|} + O\left( \#(\Gamma_8 \cap B_Y) E(Y,q) \right)
$$
as $Y \to \infty$ where
$$
E(Y,q) = \begin{cases}
e^{-c \sqrt{\log Y}} & \GN{q} \leq C \log Y \\
\GN{q}^C Y^{-\Theta} & \GN{q} > C \log Y
\end{cases}.
$$
\end{theorem}

We also have from Chapter~\ref{lall} that $\#(\Gamma_8 \cap B_Y) \asymp X^{2\delta_8}$.  In order to increase the size of the sifting set, we embed $\Gamma_8$ inside a larger subset of geodesics.  In particular, define
$$
\Xi_0 := \{a \in \Sigma_R^* : \pi(a) \in B_X\}, \hspace{.3 in} \Omega_0:=\{a \in \Sigma_R^*: \pi(a) \in B_Z\}
$$
while 
$$
\aleph_0 :=\{a \in \Sigma_8^*: \pi(a) \in B_Y\}
$$
We would like to construct a set which is the product of $\Xi_0$, $\aleph_0,$ and $\Omega_0$ such that any combination of three elements from the respective sets gives a unique geodesic.  To this end, we recall the notation from page~\pageref{conc}.  Specifically, if $a \Vert b$ denotes the admissible concatenation of two finite words.  There are two possible issues we must address:
\begin{enumerate}
\item For each $\xi \in \Xi_0$, $a \in \aleph_0$, $\omega \in \Omega_0$, the concatenation $\xi \Vert a \Vert \omega$ may not be admissible.

\item For each $\xi \in \Xi_0$, $a \in \aleph_0$, $\omega \in \Omega_0$, the concatenation $\xi \Vert a \Vert \omega$ may not be unique.  Since the length of words in $\Xi_0, \aleph, \Omega_0$ may vary it is possible that the same product may occur at different concatenation spots.
\end{enumerate}
To avoid (2), we establish a uniform length of the words in $\Omega_0$ and $\Xi_0$.  For (1), we will add a universal transition between $\xi$ and $a$ as well as $a$ and $\omega$.

In $\Gamma_R$, wordlength is commensurate with the log of the norm.  Specifically, for $a \in \Sigma_R^*$ we have $l(a) \asymp ||\pi(a)||$.  Since
$$
\# \Xi_0 = \# (\Gamma_R \cap B_X) \asymp X^{2\delta_R}
$$
there is some $l_x$ so that
$$
\Xi := \{ a \in \Sigma_R^*| a \in \Xi_0 , l(a) = l_x\} \subset \Xi_0
$$
has size  $ \gg X^{2 \delta_R}/ \log X$.  Similarly, we can find $l_z$ and
$$
\Omega := \{ a : a \in \Sigma_R^*| a \in \Omega_0, l(a) = l_z\} \subset \Omega_0
$$
of size $\gg Z^{2\delta_R}/ \log Z$.  

Finally, we address the issue that $\xi \Vert a \Vert \omega$ may not be an admissible concatenation in $\Sigma_R^*$.  Recall from Chapter~\ref{pollicott} that $(\Sigma_R, \sigma)$ is irreducible and aperiodic as long as $R>3$.  In the course of proving these properties, we showed that $\hat{f}^3(p_1) \supset p_2$ for any two states in $p_1, p_2 \in \mathcal{P}_R$ .  Therefore for any two words $a, b \in \Sigma_R^*$ there is a word $\iota$ of length $3$ such that $a \Vert \iota \Vert b \in \Sigma_R^*$.  Note that $\iota$ only depends on the final letter of $a$ and the first letter of $b$.  So we can arbitrarily choose a dictionary of three letter words $\iota$ such that $a \Vert \iota \Vert b$ is admissible and abbreviate the new word as $a \dot{\Vert} b$.  Thus, we define our sifting set as 
\begin{align*}
\Pi = \{\xi \dot{\Vert} a \dot{\Vert} \omega : \xi \in \Xi, a \in \aleph, \omega \in \Omega\}.
\end{align*}
Since $\pi(\Xi) \subset B_X, \pi(\aleph) \subset B_Y, \pi(\Omega) \subset B_Z$ there is some universal $C$ such that $\pi(\Pi) \subset B_{CN}$ where $N= XYZ$.

Write
\begin{align*}
|\mathcal{U}_\mfq| 
&= \sum_{\varpi \in \Pi} \mathbf{1}_{\{ \mtrs{\pi({\varpi})} - 4 \equiv 0 (\mfq)\}} 
= \sum_{\substack{ t \bmod\mfq \\ t^2 \equiv 4 (\mfq)}} \sum_{\varpi \in \Pi}\mathbf{1}_{\{ \mtr{\pi({\varpi})} \equiv t (\mfq)\}} \\
&= \sum_{\substack{ t \bmod\mfq \\ t^2 \equiv 4(\mfq)}} \sum_{\varpi \in \Pi} \frac{1}{\GN{\mfq}} \sum_{q | \GN{\mfq}} \sum_{\chi (q)} \chi(\mtr{\pi({\varpi})}-t)
\end{align*}
where $\displaystyle \sum_{q | \GN{\mfq}}$ is a sum over real $q$ and $\displaystyle \sum_{\chi(q)}$ ranges over all characters of $\mathbb{Z}[i]/(\mfq)$ of order $q$.  We will now separate the sum into characters of small and large order.

\section{Small Divisors}

We estimate
\begin{align*}
\mathcal{M}_\mfq := \sum_{\substack{ t \bmod\mfq \\ t^2 \equiv 4(\mfq)}} \sum_{\varpi \in \Pi} \frac{1}{\GN{\mfq}} \sum_{\substack{q | \GN{\mfq} \\ q \leq Q_0}} \sum_{\chi (q)} \chi(\mtr{\pi(\varpi)}-t)
\end{align*}
Expand from the definition of the sifting set $\Pi$ and partition $\aleph$ into residues classes in $\slt{q}$:
\begin{align*}
\mathcal{M}_\mfq 
&= \sum_{\substack{ t \bmod\mfq \\ t^2 \equiv 4(\mfq)}} \sum_{(\xi a \omega) \in \Pi} \frac{1}{\GN{\mfq}} \sum_{\substack{q | \GN{\mfq} \\ q \leq Q_0}} \sum_{\chi (q)} \chi(\mtr{\pi(\xi a \omega)}-t) \\
&= \sum_{\substack{ t \bmod\mfq \\ t^2 \equiv 4(\mfq)}} \sum_{\substack{\xi \in \Sigma\\\omega \in \Omega}} \frac{1}{\GN{\mfq}}  \sum_{\substack{q | \GN{\mfq} \\ q \leq Q_0}} \sum_{\chi (q)} \sum_{\substack{a_0 \in \slt{q}}} \chi(\mtr{\pi(\xi a_0 \omega)}-t) \left[ \sum_{\substack{a \in \aleph \\ a \equiv a_0(q)}}1\right].
\end{align*}
Now, using Theorem~\ref{fixtheta} from page~\pageref{fixtheta} we write $\mathcal{M}_\mfq= \mathcal{M}_\mfq^{(1)}+r^{(1)}(\mfq)$ where
\begin{align*}
    \mathcal{M}_\mfq^{(1)}&:=  \sum_{\substack{ t \bmod\mfq \\ t^2 \equiv 4(\mfq)}} \sum_{\substack{\xi \in \Sigma\\\omega \in \Omega}} \frac{1}{\GN{\mfq}}  \sum_{\substack{q | \GN{\mfq} \\ q \leq Q_0}} \sum_{\chi (q)}\frac{|\aleph|}{|\slt{\mfq}|} \sum_{\substack{a_0 \in \slt{q}}} \chi(\mtr{\pi(\xi) a_0 \pi(\omega)}-t) , \\
    r^{(1)}(\mfq)&:=\sum_{\substack{ t \bmod\mfq \\ t^2 \equiv 4(\mfq)}} \sum_{\substack{\xi \in \Sigma\\\omega \in \Omega}} \frac{1}{\GN{\mfq}}  \sum_{\substack{q | \GN{\mfq} \\ q \leq Q_0}} \sum_{\chi (q)} \sum_{\substack{a_0 \in \slt{q}}} \chi(\mtr{\pi(\xi) a_0 \pi(\omega)}-t) |\aleph|E(\mfq,Y).
\end{align*}
First, we address $|r^{(1)}(\mfq)|$:
$$
\sum_{\GN{\mfq}<\mathcal{Q}}|r^{(1)}(\mfq)| \ll \sum_{\GN{\mfq}<\mathcal{Q}} \tau(\mfq) \frac{|\Pi|}{\GN{\mfq}} \sum_{\substack{q | \GN{\mfq} \\ q \leq Q_0}} q^4 q^CY^{-\Theta} \ll |\Pi| (\log \mathcal{Q})^2 Q_0^CY^{-\Theta} .
$$
Back to $\mathcal{M}_\mfq^{(1)}$, we want to sum over all divisors of $\mfq$, not only the small ones.  So, we reintroduce the large divisors by writing $\mathcal{M}_\mfq^{(2)}=\mathcal{M}_\mfq^{(1)}+r^{(2)}(\mfq)$ where
\begin{align*}
    \mathcal{M}_\mfq^{(2)} &:=  \sum_{\substack{ t \bmod\mfq \\ t^2 \equiv 4(\mfq)}} \frac{|\Pi|}{\GN{\mfq}}  \sum_{\substack{q | \GN{\mfq}}} \sum_{\chi (q)}\frac{1}{|\slt{\mfq}|} \sum_{\substack{a_0 \in \slt{q}}} \chi(\mtr{a_0}-t),  \\
    r^{(2)}(\mfq) &:= \sum_{\substack{ t \bmod\mfq \\ t^2 \equiv 4(\mfq)}} \frac{|\Pi|}{\GN{\mfq}}  \sum_{\substack{q | \GN{\mfq} \\ q > Q_0}} \sum_{\chi (q)}\frac{1}{|\slt{\mfq}|} \sum_{\substack{a_0 \in \slt{q}}} \chi(\mtr{a_0}-t).
\end{align*}
Observe that we collapsed the sums over $\Sigma$ and $\Omega$ since they only reindex the sum over $\slt{q}$.

Now, define
\begin{align*}
\rho_t(\mfp) &:= \frac{1}{|\slt{\mfp}|} \sum_{\gamma \in \slt{\mfp}} \sum_{\chi(\mfp)} \chi(\mtr{\pi(\gamma)}-t) \\
&= \frac{1}{|\slt{\mfp}|} \sum_{\gamma \in \slt{\mfp}} (\GN{\mfp}-1){\bf 1}_{\{\mtr{\pi(\gamma)}=t(\mfp)\}} + (-1) {\bf 1}_{\{\mtr{\pi(\gamma)}\neq t(\mfp)\}} \\
&=\frac{\GN{\mfp}\left(\#\{ \gamma \in \slt{\mfp}: \mtr{\pi(\gamma)} = t(\mfp)\}\right) - |\slt{\mfp}|}{|\slt{\mfp}|}.
\end{align*}

\begin{lemma}  For  $t = \pm 2 \bmod \mfp$,
$$\#\{ \gamma \in \slt{\mfp}: \mtr{\pi(\gamma)} = t(\mfp)\} = \GN{\mfp}^2$$
\end{lemma}
\begin{proof}
We may as well assume $t =2$, the case when $t=-2$ is similar.  We partition $\gamma = \begin{pmatrix} a & b \\ c & d \end{pmatrix}$ into two cases according to whether or not $c = 0$:
\begin{enumerate}
\item $c=0$: In this case, we must have $ad=1$ hence $a+d=a+a^{-1}=2$.  This only happens when $x^2-2x+1$ has a root in $\mathbb{F}_\mfp$.  Examining the discriminant, we see that there is exactly one solution.  So, there's only one choice for $a$ and $d$.  There are $\GN{\mfp}$ choices for $b$.
\item $c\neq 0$: There are $\GN{\mfp}-1$ nonzero choices of $c$.  In this case, there are $\GN{\mfp}$ choices for $a$ after which $d$ is determined.  The determinant equation implies $b = c^{-1}(ad-1)$.
\end{enumerate}
Combining the two cases, we have $(\GN{\mfp}-1)\GN{\mfp} + \GN{\mfp}=\GN{\mfp}^2$ total matrices in $\slt{\mfp}$ with trace $2$.
\end{proof}
Therefore, we have 
$$\rho_t(\mfp)=\frac{\GN{\mfp}^3-(\GN{\mfp}^3-\GN{\mfp})}{\GN{\mfp}^3-\GN{\mfp}} = \frac{1}{\GN{\mfp}^2-1}.$$
We can rewrite
$$
\mathcal{M}_\mfq^{(2)} =|\Pi| \sum_{\substack{t \bmod \mfq \\ t^2 \equiv 4(\mfq)}}   \prod_{\mfp|\mfq} \frac{ \left(1+\rho_t(\mfp)\right)}{\GN{\mfp}} = |\Pi| \prod_{\mfp|\mfq} (1 + {\bf 1}_{\{\mfp\neq (1+i)\}})\frac{ \left(1+\rho_t(\mfp)\right)}{\GN{\mfp}}.
$$
Above, we used the Chinese Remainder Theorem to count $t \bmod \mfq$ satisfying $t^2 \equiv 4 (\mfq)$.  As long as $p\neq (1+i)$ we will have $\pm 2$ are distinct which gives two solutions to $t^2 \equiv 4 (\mfp)$.  If $p=(1+i)$, then we only have one solution.

Finally, we bound $\displaystyle \sum_{\mfq< \mathcal{Q}} |r^{(2)}(\mfq)|$.  Note that $|\rho_t(\mfp)| \leq \frac{1}{\GN{\mfp}}$ so
$$
|r^{(2)}(\mfq)| \ll \tau(\mfq) |\Pi| \frac{1}{\GN{\mfq}} \sum_{\substack{q|\mfq \\ q \geq Q_0}} \frac{1}{\GN{q}} \ll |\Pi| \frac{\GN{\mfq}^\epsilon}{\GN{\mfq}} \frac{1}{Q_0}.
$$
Then $\displaystyle \sum_{\GN{\mfq}<\mathcal{Q}} |r^{(2)}(\mfq)| \ll |\Pi| \frac{\mathcal{Q}^\epsilon}{Q_0}$.

\section{Large Divisors}
We first need to establish the existence of a smooth cutoff function which will allow us to extend sums over $\Xi$ and $\Omega$ to all of $\slt{\Zi}$.

\subsection{Spectral Theory of the Laplace Operator}

The Laplacian $\Delta = z^2 (\partial_{xx}+\partial_{yy}+\partial_{zz})- z\partial_{z}$ acts on $L^2(\slt{\Zi}\backslash \mathbb{H}^3)$.  There are a finite number of discrete eigenvalues in $[0,1)$:
$$
0=\lambda_0 < \lambda_1 \leq \cdots \lambda_K< 1.
$$
In $[1,\infty)$, there is the continuous spectrum as well as the remaining part of the discrete spectrum (see \cite{selberg} or \cite{sarnak}.)   For the congruence subgroup $\Gamma(q)$ we have the analogue of Selberg's $3/16$-Theorem: $\lambda_1(q) = \delta(2-\delta) \geq 3/4$ (see \cite{sarnak} or Theorem 6.1 of \cite{EGM} for proof.)   In other words, $\delta \leq 3/2$.

Denote $G = \slt{\mathbb{C}}$, $K=SU(2)$, $V:=L^2(\Gamma(q) \backslash G)$

$$
V = \bigoplus_{\lambda_j<\theta(2-\theta)} V_{\lambda_j} \oplus V^{\perp}
$$
where $V_{\lambda_j}$ is a complementary series representation of parameter $s_j$ and $V^{\perp}$ does not weakly contain any complementary series representation of parameter $s>\theta$.

The following is standard, see \cite{BK14} or \cite{KO11}.
\begin{theorem}
\label{sobnorm}
Let $\Theta$ and $(\pi,V)$ be a unitary representation of G which does not weakly contain any complementary series representation with parameter $s>\Theta$. Then for any right $K$-invariant vectors $\Psi_1, \Psi_2 \in V$
$$
|\langle \pi(g).\Psi_1, \Psi_2 \rangle| \ll ||g||^{-2(2-\Theta)} ||\Psi_1|| ||\Psi_2||
$$
as $||g||\to \infty$.
\end{theorem}

We pick a nonnegative smooth bump function $\psi$ satisfying $\int_{G/K} \psi = 1$.  The support of $\psi$ is in a ball about the identity of (tiny) radius $\eta$.  For large $X$, we define our indicator function on $\slt{\mathbb{C}}$ as
$$
\varphi_X(g) := \int_{G / K} \int_{G/K} \ind{||h_1^{-1} g h_2||<X} \psi(h_1) \psi(h_2) d h_1 d h_2.
$$
With $\eta$ small enough, we have
$$
\varphi_X(g) = \begin{cases} 1 & \text{if } ||g||<(1-\epsilon)X \\
0 & \text{if } ||g||>(1+\epsilon) X \\
\in [0,1] & \text{otherwise} \end{cases}.
$$

Since the support of $\varphi_X$ is within a ball of radius $(1+\epsilon)X$, we have the following
$$
\sum_{\gamma \in \slt{\mathbb{Z}[i]}} \varphi_X(\gamma) \ll \# \{ \gamma \in \slt{\mathbb{Z}[i]} : ||\gamma||<C_1X\} \ll X^4
$$
Now, we wish to establish that $\varphi_X$ assigns roughly equal weights to residue classes in $\slt{p}$.  
\begin{proposition}
\label{indgap}
For squarefree $q$ and $\gamma_0 \in \slt{q}$, we have
$$
\sum_{\substack{ \xi \in \slt{\mathbb{Z}[i]} \\ \xi \equiv \gamma_0(q) }} \varphi_X (\xi) = \frac{1}{|\slt{q}|} \sum_{\xi \in \slt{\mathbb{Z}[i]}} \varphi_X(\xi) + O(X^{3}).
$$
\end{proposition}

We follow along with the proof found in \cite{BK14}.

\begin{proof}
\begin{align*}
\sum_{\substack{ \xi \in \slt{\mathbb{Z}[i]} \\ \xi \equiv \gamma_0(q) }} \varphi_X (\xi)  &= \sum_{\gamma \in \Gamma(q)} \varphi_X(\gamma\gamma_0) \\ &= \int_{G / K}  \int_{G/K} \sum_{\gamma \in \Gamma(q)}  \ind{||h_1^{-1} \gamma \gamma_0 h_2||<X} \psi(h_1)   \psi(h_2) d h_1  d h_2.
\end{align*}
Define
\begin{align*}
\mathcal{F}_q(h,g) &= \sum_{\gamma \in \Gamma(q)} \ind{||h^{-1}\gamma g||<X}, \\
\Psi_q(g) &= \sum_{\gamma \in \Gamma(q)} \psi(\gamma g) , \\
\Psi_{q,\gamma_0}(g) &= \sum_{\gamma \in \Gamma(q)} \psi(\gamma_0^{-1} \gamma g)
\end{align*}
so that we have the following identity (via a standard folding and unfolding argument):
\begin{align*}
\sum_{\substack{ \xi \in \slt{\mathbb{Z}[i]} \\ \xi \equiv \gamma_0(q) }} \varphi_X (\xi)  &= \langle \mathcal{F}_q, \Psi_q \otimes \Psi_{q,\varphi_0} \rangle_q \\
&=\int_{G/K} \ind{||h||<X} \langle \pi(h).\Psi_q,\Psi_{q,\gamma_0}\rangle_q dh
\end{align*}
Where $\langle \cdot, \cdot \rangle_q$ denotes the inner product on $L^2(\Gamma(q)\backslash G/K)$.  Now, decompose the matrix coefficient into the following
\begin{align*}
&\int_{G/K} \ind{||h||<X} \langle \pi(h).\Psi_q|_V,\Psi_{q,\gamma_0}|_V\rangle_q dh \\ &+ \int_{G/K} \ind{||h||<X} \langle \pi(h).\Psi_q|_{V^\perp},\Psi_{q,\gamma_0}|_{V^\perp}\rangle_q dh
\end{align*}
where $|_V$ denotes projection onto the old forms and $V^{\perp}$ is orthogonal.

Claim 1: $\langle \pi(h).\Psi_q|_V,\Psi_{q,\gamma_0}|_V\rangle_q = \frac{1}{[\Gamma(q):\Gamma]}\langle \pi(h).\Psi_1|_V,\Psi_{1,\gamma_0}|_V\rangle_1$

If $\varphi_1^{(q)}, \ldots, \varphi_l^{(q)}$ denote the oldforms for the spectrum below $\Theta(2-\Theta)$, then we can rewrite them as normalizations of the eigenfunctions at level one:
$$
\varphi_i^{(q)} = \frac{1}{\sqrt{[\Gamma(q):\Gamma]}} \varphi_i^{(1)}
$$
The rest follows from folding and unfolding.

Claim 2: $\int_{G/K} \ind{||h||<X} \langle \pi(h).\Psi_q|_{V^\perp},\Psi_{q,\gamma_0}|_{V^\perp}\rangle_q dh \ll X^{3}$

Since $\eta$ is fixed and $\psi$ has 'bounded' support, we have $\mathcal{S} \Psi_q|_{V^\perp}  \leq \mathcal{S} \Psi \ll 1$ (and similar for $\Psi_{q,\gamma_0}$).  So, Theorem~\ref{sobnorm} gives
$$
\int_{G/K} \ind{||h||<X} ||h||^{2\delta-4} dh \ll X^{3}.
$$

Thus, we have shown
$$
\sum_{\substack{ \xi \in \slt{\mathbb{Z}[i]} \\ \xi \equiv \gamma_0(q) }} \varphi_X (\xi) = \frac{1}{[\Gamma(q):\Gamma]}\langle \pi(h).\Psi_1|_V,\Psi_{1,\gamma_0}|_V\rangle_1+ O(X^{3}).
$$
When we use the same argument, setting $q = 1$, the theorem follows from 
$$
\sum_{\substack{ \xi \in \slt{\mathbb{Z}[i]} }} \varphi_X (\xi) = \langle \pi(h).\Psi_1|_V,\Psi_{1,\gamma_0}|_V\rangle_1+ O(X^{3})
$$
\end{proof}

\subsection{Character Sums}\

\begin{proposition}
\label{kloo}
Suppose $q$ is a rational prime and $\mfq$ lies over $q$, $\chi: \Zi \to S^1$ has order $q$ and $\xi \in \mathbb{Z}[i]^4$ satisfies $(\xi,\mfq)=1$ (component-wise.)  We have
$$
\left|\sum_{s \in \slt{\mfq}} \chi_q(s \cdot \xi) \right| \ll \GN{\mfq}^{3/2}.
$$
\end{proposition}
\begin{proof}
We may as well assume $y \neq 0 (\mfq)$, otherwise we could alter the following argument on $s$.  

We now partition all $s = \begin{pmatrix} a & b \\ c & d \end{pmatrix}$ into two cases, either $c=0$ or $c\neq 0$.  If $c\equiv 0 (\mfq)$, then $ad-bc=1$ implies $d=a^{-1} (\mfq)$ and $b$ is anything. Since $b$ ranges over all $\Zi/(\mfq)$, we have
\begin{align*}
\sum_{\substack{s \in \slt{\mfq}\\c=0 (\mfq)}} \chi_q(s \cdot \xi) &= \sum_{a(\mfq)}'\sum_{b(\mfq)} \chi_q(ax + by + a^{-1}w)\\& = \sum_{a(\mfq)}' \chi_q(ax+a^{-1}w) \sum_{b(\mfq)} \chi_q(by) =0.
\end{align*}
On the other hand if $c\neq 0$, then we can pick any $a,d$ which imply $b = c^{-1}(ad-1)$ so
\begin{align*}
\sum_{\substack{s \in \slt{\mfq}\\c\neq 0(\mfq)}} \chi_q(s \cdot \xi) &= \sum_{c(\mfq)}'\sum_{a,d(\mfq)} \chi_q(ax+c^{-1}(ad-1)y+cz+dw) \\
&=\sum_{c(\mfq)}' \chi_q(cz-c^{-1}y) \sum_{a(\mfq)} \chi_q(ax) \sum_{d(\mfq)} \chi_q(d(c^{-1}ay+w)).
\end{align*}
Now, note that
$$
\sum_{d(\mfq)} \chi_q(d(c^{-1}ay+w)) = \begin{cases} 0 & a \not\equiv -cy^{-1}w (\mfq) \\ \GN{\mfq} & a \equiv -cy^{-1}w (\mfq) \end{cases}.
$$
Hence
\begin{align*}
    \sum_{\substack{s \in \slt{\mfq}\\c\neq 0(\mfq)}}\chi_q(s \cdot \xi) &= \GN{\mfq}  \sum_{c(\mfq)}' \chi_q(cz-c^{-1}y) \chi_q(-cy^{-1}wx) \\
    &=\GN{\mfq}  \sum_{c(\mfq)}' \chi_q(c(z-y^{-1}wx)-c^{-1}y).
\end{align*}
Now, since $y\neq 0$, we have a nontrivial Kloosterman sum (or perhaps a Ramanujan sum if $z-y^{-1}wx=0$.)  Regardless of whether $q$ is split or inert, we have
$$
|K(\chi_q;a,b)| \leq 2\GN{q}^{1/2}.
$$
When $q$ splits, this is just the Weil bound.  When $q$ is inert, see Theorem 5.45 of \cite{lidln}.
\end{proof}

From Proposition~\ref{kloo}, we have the following
\begin{corollary}
For square-free $\mfq$, $\chi$ a character of $\Zi/ (\mfq)$ of order $\GN{\mfq}$ and $(\xi,\mfq)=1$, we have
$$
\left|\sum_{s \in \slt{\mfq}} \chi_q(s \cdot \xi) \right| \ll \GN{\mfq}^{3/2}.
$$
\end{corollary}
\begin{proof}
Pontryagin duality allows us to express $\chi_q$ as a product of characters of prime order.  Apply Proposition~\ref{kloo} to each term.
\end{proof}

Now, we combine this bound with our indicator function $\varphi$:
\begin{proposition}
For square-free $q$ and a $\chi_q$ character of order $\GN{q}$, we have
$$\left| \sum_{\xi \in \slt{\Zi}} \varphi_X(\xi) \chi_q(s \cdot \xi) \right| \ll \GN{\mfq}^{-3/2+\epsilon}X^4+\GN{q}^3X^3.$$
\end{proposition}
\begin{proof}
Partition over residue classes over $\mfq$ and apply Proposition~\ref{indgap}:
\begin{align*}
    \left| \sum_{\xi \in \slt{\Zi}} \varphi_X(\xi) \chi_q(s \cdot \xi) \right| &= \left| \sum_{\gamma \in \slt{\mfq}} \chi_q(\gamma \cdot s) \sum_{\substack{\xi \in \slt{\Zi} \\ \xi \equiv \gamma(\mfq)}} \varphi_X(\xi)\right| \\
    &= \left|\sum_{\gamma \in \slt{\mfq}} \chi_q(\gamma \cdot s)\right| \frac{X^4}{|\slt{\mfq}|} + O(\GN{\mfq}^3X^3) \\
    &\ll \frac{\GN{\mfq}^{3/2+\epsilon} X^4}{\GN{q}^3} + \GN{\mfq}^3X^3.
\end{align*}
\end{proof}

\subsection{Large Divisors}\

Define
\begin{align*}
    r(\mfq) := \sum_{t^2 \equiv 4 (\mfq)} \sum_{\varpi \in \Pi} \frac{1}{\GN{\mfq}} \sum_{\substack{q | \GN{\mfq} \\ q \geq Q_0}} \sum_{\chi (q)} \chi(\mtr{\pi(\varpi)} - t).
\end{align*}
Our goal is to bound
\begin{align*}
    \mathcal{E}:= \sum_{\GN{\mfq} <\mathcal{Q}} |r(\mfq)| = \sum_{\GN{\mfq} <\mathcal{Q}} r(\mfq)\zeta(\mfq),
\end{align*}
where the sum ranges over square-free Gaussian ideals up to norm $\mathcal{Q}$ and $\zeta(\mfq)$ captures the argument of $r(\mfq)$.  In order to interchange the order of $q$ and $\mfq$ in $\mathcal{E}$ we introduce
$$
\zeta_1(q,r) := q \sum_{\substack{\GN{\mfq}<\mathcal{Q}\\ \GN{\mfq} \equiv 0 (q)}} \frac{\zeta(\mfq)}{\GN{\mfq}} \sum_{t^2 \equiv 4 (\mfq)} \chi_q^r(-t) {\bf 1}_{\{\chi_q^r\in \widehat{\mathbb{Z}[i]/(\mfq)} \}}
$$
for real $q$ and $r \in (\mathbb{Z}/q\mathbb{Z})^*$ .  Insert $\zeta_1$ into $\mathcal{E}$:
$$
\mathcal{E}=\sum_{Q_0 \leq q < \mathcal{Q}} \frac{1}{q} \sum_{r(q)}' \sum_{\varpi \in \Pi} \chi_q^r(\mtr{\pi(\varpi)}) \zeta_1(q,r)
$$
where the sum is over all square-free $q \in [Q_0,\mathcal{Q}]$.

We now restrict our attention to short intervals in $q$ and fixed $a \in \aleph$, i.e.
\begin{align*}
\mathcal{E}_1(Q,a) &:= \sum_{q \asymp Q} \left|\sum_{r (q)}' \zeta_1(q,r)  \sum_{\xi \in \Sigma}\sum_{\omega \in \Omega} \chi_q^r(\mtr{\pi(\xi a \omega)}) \right| \\
&=\sum_{q \asymp Q} \zeta_2(q) \sum_{r (q)}'\zeta_1(q,r)  \sum_{\xi \in \Sigma}\sum_{\omega \in \Omega} \chi_q^r(\mtr{\pi(\xi a \omega)}) 
\end{align*}
where we introduced $\zeta_2(q)$ to capture the absolute value of each term.  Now apply Cauchy-Schwarz in the $\xi$ parameter:
\begin{align*}
    |\mathcal{E}_1|^2 \ll |\Sigma| \sum_{\xi \in \Sigma} \left| \sum_{q \asymp Q} \zeta_2(q) \sum_{r(q)}' \zeta_1(q,r) \sum_{\omega \in \Omega} \chi(\mtr{\pi(\xi a \omega)}) \right|^2.
\end{align*}
The support of $\pi(\Sigma)$ is within $B_X$, so we replace sequences in $\Sigma$ with matrices in $B_X$ as follows
\begin{align*}
    |\mathcal{E}_1|^2 \ll |\Sigma| \sum_{\gamma \in \slt{\Zi}} \varphi_X(\gamma) \left| \sum_{q \asymp Q} \zeta_2(q) \sum_{r(q)}' \zeta_1(q,r) \sum_{\omega \in \Omega} \chi_q^r(\gamma \cdot {\pi(a \omega)}) \right|^2
\end{align*}
where we have replaced the trace of $\pi(\xi a \omega)$ with the dot product $\gamma \cdot \pi(a \omega)$ since 
$$
\mtr{\pi(\xi a \omega)} = \sum_{i,j} \pi(\xi)_{i,j} \pi(a \omega)_{i,j}.
$$
Opening the square, we have
\begin{align*}
    |\mathcal{E}_1|^2 &\ll |\Sigma| \sum_{q, q'} \sum_{\substack{ r (q) \\ r' (q')}}' \zeta_2(q) \zeta_1(q,r) \overline{\zeta_2(q') \zeta_1(q', r')} \\
&{\color{white}{.}} \hspace{.5in} \sum_{\omega, \omega'} \sum_{\gamma \in \slt{\Zi}} \varphi_X(\gamma) \chi_q^r(\gamma \cdot \pi(a\omega)) \chi_{q'}^{r'}(-\gamma \cdot \pi(a\omega')) \\
    &\ll \mathcal{Q}^\epsilon |\Sigma| \sum_{q, q'} \sum_{\substack{ r (q) \\ r' (q')}}' \sum_{\omega, \omega'} \left| \sum_{\gamma \in \slt{\Zi}} \varphi_X(\gamma) \chi_q^r(\gamma \cdot \pi(a\omega)) \chi_{q'}^{r'}(-\gamma \cdot \pi(a\omega')) \right|.
\end{align*}
We would like to combine $\chi_q$ and $\chi_{q'}$.  Let $\hat{q}$ be the least common multiple $[q,q']$ and $q_1 = \hat{q}/q$, $q_1' = \hat{q}/q'$ be the primes distinct to $q$ and $q'$ respectively.  If $b\equiv (q_1)^{-1} (q')$ and $b' \equiv (q_1')^{-1} (q)$ we have 
$$
\chi_q^r(\gamma \cdot \pi(a\omega)) \chi_{q'}^{r'}(-\gamma \cdot \pi(a\omega')) = \chi_{\hat{q}}(\xi \cdot (r q_1'b' \pi (a\omega) - r'q_1b \pi(a\omega'))).
$$
If $s:=r q_1'b' \pi (a\omega) - r'q_1b \pi(a\omega')$, then we may not have $(s,\hat{q})=1$.  So, remove common factors to get $(s',q_0)=1$.  If $\hat{q_0}=\hat{q}/q_0$ represents all the factors we removed from $\hat{q}$, then notice
$$
rq_1'b' \pi(a\omega) \equiv r'q_1b \pi(a \omega') \bmod \hat{q_0}.
$$
Multiply (on the left) by $\pi(a)^{-1} \in \slt{\hat{q_0}}$ to remove $\pi(a)$.  Take determinants of both sides and since $\Det{\pi(\omega)}=\Det{\pi(\omega')}=1$, we get
$$
(rq_1'b')^2 \equiv (r'q_1b)^2 \bmod \hat{q_0}.
$$
Now, $\hat{q_0} | (q,q')$ so $(q_1,\hat{q_0}) = (r',\hat{q_0}) = (b,\hat{q_0})=1$.  Similarly, $(q_1'b'r,\hat{q_0})=1$.  So, we can find some $u$ such that
$$
q_1'b'r \equiv u q_1br' \bmod \hat{q_0}
$$
which satisfies $u^2 \equiv 1 (\hat{q_0})$ by the squared congruence relation above - there are only $N^\epsilon$ such $u$.  We also have
$\pi(\omega) \equiv u\pi(\omega') \bmod \hat{q_0}$.

Back to our estimate of $\mathcal{E}_1$, partition the sum over $q$ and $q'$ via their least common multiple as follows:
\begin{align*}
    |\mathcal{E}_1(Q,a)|^2 \ll |\Sigma| &\sum_{\hat{q} \asymp Q} \sum_{q_1q_1'\tilde{q_0}\hat{q_0} = \hat{q}} \sum_{\substack{u (\hat{q_0}) \\ u^2 \equiv 1 (\hat{q_0})}} \sum_{r(q)}'\sum_{\substack{r'(q')\\q_1'b'r \equiv uq_1br'(\hat{q_0})}} \\
&\sum_{\omega' \in \Omega} \sum_{\substack{\omega \in \Omega \\\pi(\omega) \equiv u \pi(\omega') (\hat{q_0})}} \left| \sum_{\xi \in \slt{\Zi}} \varphi_X(\xi) \chi_{q_0}(\xi \cdot s) \right|.
\end{align*}
Since the support of $\pi(\Omega)$ is in $B_Z$, we replace the second sum over $\Omega$ with one over $\slt{\Zi}$ via $\varphi_Z$:
\begin{align*}
    |\mathcal{E}_1(Q,a)|^2 \ll |\Sigma| &\sum_{\hat{q} \asymp Q} \sum_{q_1q_1'\tilde{q_0}\hat{q_0} = \hat{q}} \sum_{\substack{u (\hat{q_0}) \\ u^2 \equiv 1 (\hat{q_0})}} \sum_{r(q)}'\sum_{\substack{r'(q')\\q_1'b'r \equiv uq_1br'(\hat{q_0})}} \\
& \sum_{\omega' \in \Omega} \sum_{\substack{\beta \in \slt{\Zi} \\\ \beta \equiv u \pi(\omega') (\hat{q_0})}} \varphi_Z(\beta) \left| \sum_{\xi \in \slt{\Zi}} \varphi_X(\xi) \chi_{q_0}(\xi \cdot s) \right|.
\end{align*}
Now, we apply our bounds for the smoothing function $\varphi_X$:
\begin{align*}
    |\mathcal{E}_1(Q,a)|^2 
    \ll |\Sigma| \sum_{\hat{q} \asymp Q} \sum_{q_1q_1'\tilde{q_0}\hat{q_0} = \hat{q}} \sum_{\substack{u (\hat{q_0}) \\ u^2 \equiv 1 (\hat{q_0})}} \sum_{r(q)}' &\sum_{\substack{r'(q')\\q_1'b'r \equiv uq_1br'(\hat{q_0})}} |\Omega| \left[ \frac{Z^4}{\GN{\hat{q}_0}^3}+Z^3 \right] \cdot \\
& \left[ \GN{q_0}^{-3/2+\epsilon}X^4+\GN{q_0}^3X^3 \right].
\end{align*}
There are at most $\GN{q}$ choices for $r$ and then $\frac{\GN{q'}}{\GN{\hat{q}_0}}$ choices for $r'$.  Therefore, 
\begin{align*}
    |\mathcal{E}_1(Q,a)|^2 
    \ll |\Sigma| \sum_{\hat{q} \asymp Q} N^\epsilon \sum_{q_0\hat{q_0}=\hat{q}}& \frac{Q^2\GN{q_0}}{\GN{\hat{q}}}  |\Omega| \left[ \frac{Z^4}{\GN{\hat{q}_0}^3}+Z^3 \right]\cdot \\ &\left[ \GN{q_0}^{-3/2+\epsilon}X^4+\GN{q_0}^3X^3 \right] .
\end{align*}
Now, we use the fact that $|\Omega| \asymp \frac{Z^{2\delta}}{\log Z}$ and similarly $|\Sigma| \asymp \frac{X^{2\delta}}{\log X}$ to insert $|\Omega|Z^{-2\delta}$ and $|\Sigma|X^{-2\delta}$:
\begin{align*}
    |\mathcal{E}_1(Q,a)|^2 
    \ll N^\epsilon &Q^2 |\Sigma|^2|\Omega|^2 (XZ)^{2(2-\delta)}  \\
 &\times \left( \sum_{Q \ll \GN{\hat{q}} \ll Q^2} \frac{1}{\GN{\hat{q}}} \left[ \frac{1}{\GN{\hat{q}}^{1/2}}+\frac{1}{Z}+\frac{Q^8}{X} \right] \right).
\end{align*}
So, we get the following
\begin{theorem}
$$|\mathcal{E}_q(Q,a)| \ll N^\epsilon Q |\Sigma| |\Omega| (XZ)^{2-\delta} \left[ \frac{1}{Q^{1/4}}+\frac{1}{Z^{1/2}}+\frac{Q^4}{X^{1/2}} \right].$$
\end{theorem}
Summing over $a$ and $Q$ gives
\begin{theorem}
$$\mathcal{E} \ll N^\epsilon |\Pi| (XZ)^{2-\delta} \left[ \frac{1}{Q_0^{1/4}}+\frac{1}{Z^{1/2}}+\frac{\mathcal{Q}^4}{X^{1/2}}\right].$$
\end{theorem}

\section{Sieve Theorem}

We have
$$
|U_q| = \beta(q) |\Pi| + r^{(1)}(q) + r^{(2)}(q) + r(q)
$$
where
\begin{align*}
\sum_{\GN{q} < Q} |r^{(1)}(q)| &\ll |\Pi| (\log Q)^2 \left( e^{-c \sqrt{\log Y}} + Q_0^C Y^{-\Theta}\right), \\
\sum_{\GN{q} < Q} |r^{(2)}(q)| &\ll |\Pi| \frac{Q^\epsilon}{Q_0}, \\
\sum_{\GN{q} < Q} |r(q)| &\ll N^\epsilon |\Pi| (XZ)^{2-\delta} \left( \frac{1}{Q_0^{1/4}} + \frac{1}{Z^{1/2}} + \frac{Q^4}{X^{1/2}} \right).
\end{align*}
Recall that $c$, $C$ and $\Theta$ are fixed constants coming from Theorem~\ref{fixtheta} page~\pageref{fixtheta}.  By our construction of the sifting set (in which $\Xi$ and $\Omega$ depend on $R$ but $\aleph$ does not), as we send $R \to \infty$ we can get $\delta_R$ arbitrarily close to 2 while $c$, $C$, and $\Theta$ remain constant.

For $r^{(1)}(q)$ and $r^{(2)}(q)$ we need $y>0$, $\alpha_0>0$ and $\alpha_0 C < y \Theta$, where $x+y+z=1$.  For $r(q)$ we need
\begin{align*}
\alpha_0 / 4 &> (x+z)(2-\delta), \\
z/2 &> (x+z)(2-\delta), \\
x/2 &> (x+z)(2-\delta) + 8\alpha.
\end{align*}
Observe from the last inequality that taking $x$ near $1$ and $\delta_R$ near $2$, we must have $\alpha<1/16$.  In order to achieve this level of distribution, we set $\alpha = 1/16 - \eta$ and assume for now that $\delta> 2 - \eta$.  In order to have $|\Pi| \gg N^{2\delta - \eta}$ we set $x = 1 - \eta$.

For the remaining parameters, we set
$$
z = \frac{\eta}{1+C/\Theta}, \hspace{.2cm} y = \frac{zC}{\Theta}, \hspace{.2cm} \alpha_0 = \frac{5z}{6}
$$
so that $\alpha_0 C<y \Theta$.  Moreover if $\delta>2-z/5= 2 - \frac{\eta}{5(1+C/\Theta)}$ we have
$$
z/2> \alpha_0/4 =  \frac{5z}{24}> \frac{z}{5} > 2 - \delta>(x+z)(2-\delta).
$$
Hence the three inequalities for $r(q)$ are satisfied.  We get a power savings in each error term except $r^{(1)}$ where we save an arbitrary power of log.  This proves the level of distribution for the following
\begin{theorem}
\label{sieve}
For sufficiently small $\eta$, there is a large enough $R$ so that $U$ has level of distribution $Q=T^{1/16-\eta}$.  In other words, there exists a multiplicative function $\beta: \Zi \to \mathbb{R}$ satisfying
$$
\prod_{\substack{p \\ w \leq \GN{p}< z}} (1- \beta(p))^{-1} \leq  C \left( \frac{\log z}{ \log w} \right)^2
$$
for any $2 \leq w <z$ and a decomposition
$$
|U_q| = |\Pi| \beta(q) + r(q)
$$
so that for all $K$
$$
\sum_{\substack{\text{square-free } q \\ \GN{q}<Q}} |r(q)| \ll_K \frac{|\Pi|}{\log^K N}.
$$
Moreover, when
$$
X=N^{1-\eta}
$$
we have
$$
|\Pi| \gg N^{2\delta-\eta}.
$$
\end{theorem}

\begin{proof}
We must show that the sieve dimension is 2.  The following summation formulas for primes in arithmetic progressions are consequences of Mertens work in \cite{merten}:
\begin{align*}
\sum_{\substack{p \leq n\\p \equiv 1 (4)}} \frac{1}{p} &= \frac{1}{2} \log \log n + B_1 + O\left(\frac{1}{\log n}\right), \\
\sum_{\substack{p \leq n\\p \equiv 3 (4)}} \frac{1}{p} &= \frac{1}{2} \log \log n + B_3 + O\left(\frac{1}{\log n}\right).
\end{align*}
For Gaussian primes, we only need the first equality:
$$
\sum_{\GN{\mfp} \leq n} \frac{1}{\GN{\mfp}} = \frac{1}{2} + \sum_{\substack{ p \leq n \\ p \equiv 1 (4)}} \frac{2}{p} + \sum_{\substack{p \leq \sqrt{n} \\ p \equiv 3 (4)}} \frac{1}{p^2} = \log \log n + O(1).
$$

Since $\beta(\mfp) \in (0,1)$ for all $p$, we have
\begin{align*}
\prod_{w \leq \GN{\mfp} \leq z} (1-\beta(\mfp))^{-1} &= \exp\left(-\sum_{w \leq \GN{\mfp} \leq z} \log(1-\beta(\mfp))\right) \\
&=\exp\left( \sum_{w \leq \GN{\mfp} \leq z} \sum_{k=1}^\infty \frac{\beta(\mfp)^k}{k}\right) \\
&= \exp\left( \sum_{w \leq \GN{\mfp} \leq z} \beta(\mfp)\right)\exp\left( \sum_{w \leq \GN{\mfp} \leq z} \sum_{k>1} \frac{\beta(\mfp)^k}{k}\right).
\end{align*}

The second exponential is negligible since
\begin{align*}
    S=\sum_{\mfp} \sum_{k>1} \frac{\beta(\mfp)^k}{k} &\leq \frac{1}{2} \sum_{\mfp} \beta(\mfp)^2 \sum_{k=0}^\infty \beta(\mfp)^k \\
    &< \frac{1}{2} \sum_{\mfp} \frac{\beta(\mfp)^2}{ (1-\beta(\mfp))} \\
    &\ll \sum_{p \equiv 1 (4)}2\frac{b(p)^2}{1-b(p)} + \sum_{p \equiv 3 (4)}\frac{b(p^2)^2}{1-b(p^2)},
\end{align*}
where we have defined
$$
b(x) = \frac{2}{x} \left(1+ \frac{1}{x^2-1}\right).
$$
Since $\sum_n b(n)^2$ converges and $(1-\beta(n))^{-1} \to 1$ partial summation gives $S<\infty$.  Hence
\begin{align*}
\prod_{w \leq \GN{\mfp} \leq z} (1-\beta(\mfp))^{-1} &\leq C \exp\left( \sum_{w \leq \GN{\mfp} \leq z} \beta(\mfp)\right)  \\
&< C \exp \left( \sum_{w \leq \GN{\mfp} \leq z} \frac{2}{\GN{\mfp}} \right) \\
&< C \exp\left(2 \log \left( \frac{\log z}{\log w} \right) \right).
\end{align*}
\end{proof}

From Theorem~\ref{sieve} and the Fundamental Lemma of sieve theory (see Lemma 6.3 of \cite{IK}) we have the following:
\begin{theorem}
\label{brunsieve}
Define 
$$
\Pi_{AP} :=\{\varpi \in \Pi : p | (\mtrs{\pi(\varpi)}-4) \implies \GN{p}>N^{\alpha}\}.
$$
We have
$$
\# \Pi_{AP} > N^{2\delta - \eta}.
$$
\end{theorem}

\section{Final Estimates}

We have the lower bound
$$
\# \Pi_{AP} =\#\{\varpi \in \Pi : p | (\mtrs{\pi(\varpi)}-4) \implies \GN{p}>N^{\alpha}\} >N^{2\delta - \eta}.
$$
On the other hand,
\begin{align*}
    \#\{\gamma \in \Gamma_R &\cap B_N : (\mtrs{\pi(\gamma)}-4) \text{ is square-free}\} \\
 &\geq \#\{\gamma \in \Pi_{AP} : (\mtrs{\pi(\gamma)}-4) \text{ is square-free}\} \\
    &> N^{2\delta-\eta} - \#\{\gamma \in \Pi_{AP} : (\mtrs{\pi(\gamma)}-4) \text{ is not square-free}\} .
\end{align*}
We will examine the last term 
$$\Pi^\Box_{AP}:=\{\gamma \in \Pi_{AP} : (\mtrs{\pi(\gamma)}-4) \text{ is not square-free}\}$$
more closely.  If $\gamma \in \Pi_{AP}^\Box$, then we can find a prime $p$ with $p^2 | (\mtrs{\pi(\gamma)}-4)$.  On one hand, $\mtrs{\pi(\gamma)}-4$ factors as $\mtr{\pi(\gamma)}\pm 2$ so $\GN{p} \ll N^{1/2}$.  On the other hand, $\Pi_{AP}^\Box \subset \Pi_{AP}$ implies $\GN{p}>N^\alpha$.  Therefore,
\begin{align*}
    \#\Pi_{AP}^\Box &\leq \sum_{N^\alpha < \GN{p} \ll N^{1/2}} \sum_{\substack{\GN{t}<N \\t^2 -4 \equiv 0 (p^2)}} \#\{\gamma \in \Gamma_R \cap B_N : \mtr{\pi(\gamma)}=t\} \\
    &\ll \sum_{N^\alpha < \GN{p} \ll N^{1/2}} \frac{N^2}{\GN{p}^2} N^{2+\epsilon}\ll N^{4-\alpha+\epsilon},
\end{align*}
where we trivially bounded the trace multiplicity, i.e.
$$
\#\{\gamma \in \Gamma_R \cap B_N : \mtr{\pi(\gamma)}=t\}< \#\{s \in B_N : \mtr{s}=t\} \ll N^{2+\epsilon}.
$$
So, as long as $2\delta-\eta > 4-\alpha+\epsilon$, we have:

\begin{theorem}
\label{sqfr}
There is an $R$ such that as $N \to \infty$ we have
$$
\#\{\gamma \in \Gamma_R \cap B_N : \mtrs{\pi(\gamma)}-4 \text{ is square-free}\} >N^{2\delta_R-\eta}.
$$
\end{theorem}

 We now select discriminants which contribute the most to the count above.  Define
$$
\mathcal{T}:=\{t \in \Zi : t^2-4 \text{ is square-free}\}
$$
and 
$$
\mathcal{M}_R(t) :=\#\{ \gamma \in \Gamma_R : \mtr{\pi(\gamma)} = t\}.
$$

\begin{proposition}
For any $\eta$ there is a large enough $R$ so that
$$
\sum_{\substack{t \in \mathcal{T} \\ \GN{t} \leq N}} {\bf 1}_{\{\mathcal{M}_R(t) \geq \GN{t}^{2\delta_R -2 - 2\eta}\}} > N^{2\delta_R-2-\eta}.
$$
\end{proposition}
\begin{proof}
The previous theorem gives 
\begin{align*}
N^{2\delta_R-\eta}&<\sum_{\substack{t \in \mathcal{T} \\ \GN{t} \leq N}} \mathcal{M}_{R,N}(t) \\ &=\sum_{\substack{t \in \mathcal{T} \\ \GN{t} \leq N}} \mathcal{M}_{R,N}(t) \left( {\bf 1}_{\{\mathcal{M}_{R,N}(t) \geq W\}}+{\bf 1}_{\{\mathcal{M}_{R,N}(t) < W\}} \right)
\end{align*}
where $W$ is a parameter of our choice.  Trivially, we have $\mathcal{M}_{R,N}(t) \ll N^{1+\epsilon}$.  So,
$$
N^{2\delta_R-\eta}\ll N^{2+\epsilon} \sum_{\substack{t \in \mathcal{T} \\ \GN{t} \leq N}} {\bf 1}_{\{\mathcal{M}_{R,N}(t) \geq W\}} + N^2W.
$$
Now, set $W = N^{2\delta_R-2-2\eta}$ and the claim follows.
\end{proof}

For any $\epsilon>0$, we can find $\eta$ small and $R$ large so that
$$
2\delta_R - 2 - \eta > 2 - \epsilon.
$$
The choice of $R$ gives our compact region (geodesics do not visit the cusp when their symbolic encodings have small entries.)  We define
$$
\mathcal{D} := \{D = t^2-4 : t \in \mathcal{T}, \mathcal{M}_R(t) > \GN{t}^{2\delta_R-2-2\eta} \}
$$
which gives a subset of all fundamental discriminants. The previous claim gives us a lower bound for the number of these discriminants:
\begin{align*}
\#\{d \in \mathcal{D} : \GN{d} \leq T\} &\geq \#\{ t \in \mathcal{T}: \GN{t} \leq T^{1/2}, \mathcal{M}_R(t) >\GN{t}^{2\delta_R-2-2\eta}\} \\
&> (T^{1/2})^{2\delta_R-2-\eta} \\
&>T^{1-\epsilon}.
\end{align*}
Moreover, for each discriminant $d \in \mathcal{D}$,if $d=t^2-4$ then
$$
\mathcal{M}_R(t) > \GN{t}^{2-\epsilon}> \sqrt{\GN{D}}^{2-\epsilon} \gg |C_D|^{1-\epsilon'}
$$

Thus, after renaming constants we have proved Theorem~\ref{main}, which we restate here:
\begin{theorem}
For any $\epsilon>0$, there is a compact region $Y(\epsilon) \subset \Gamma \backslash \mathbb{H}^3$ and a set $D(\epsilon)$ of fundamental discriminants such that
\begin{align*}
\#\{ D \in \mathcal{D}(\epsilon): \GN{D}<X \} &\gg_\epsilon X^{1-\epsilon}, \hspace{.3in} X \to \infty
\end{align*}
and for all $D \in \mathcal{D}(\epsilon)$, 
\begin{align*}
\#\{ \gamma \in C_D : \gamma \subset Y(\epsilon)\} &> |C_D|^{1-\epsilon}.
\end{align*}
\end{theorem}

\bibliographystyle{alpha}
\bibliography{total_copy}

\end{document}